\documentclass[11pt]{article}
\usepackage[english]{babel}%
\usepackage[letterpaper,top=2cm,bottom=2cm,left=3cm,right=3cm,marginparwidth=1.75cm]{geometry}
\usepackage{amsmath}
\usepackage{amssymb}
\usepackage{amsthm}
\usepackage{paralist}
\usepackage[misc]{ifsym}
\usepackage{epsfig} %
\usepackage{epstopdf} %
\usepackage[colorlinks=true]{hyperref}
\allowdisplaybreaks

\newtheorem{theorem}{Theorem}[section]
\newtheorem{corollary}[theorem]{Corollary}
\newtheorem{lemma}[theorem]{Lemma}
\newtheorem{proposition}[theorem]{Proposition}

\newtheorem{definition}[theorem]{Definition}
\newtheorem{remark}[theorem]{Remark}

\usepackage{authblk}

\newcommand{\ind}{\,\mbox{d}}

\title{Characterizing nonlinear information in the linear sampling method for inverse medium scattering} %
\author{Lorenzo Audibert\thanks{
PRISME, EDF R\&D, 6 Quai Watier, 78400 Chatou, France and UMA, Inria, ENSTA Paris, Institut Polytechnique de Paris, 91120 Palaiseau, France. \texttt{(lorenzo.audibert@edf.fr})} ~and~ 
Shixu Meng\thanks{Department of Mathematical Sciences, The University of Texas at Dallas, 75080 Richardson,  USA. (\texttt{smeng@utdallas.edu})} }

\begin{document}
\maketitle

\begin{abstract}
This work is concerned with the nonlinear information in the linear sampling method for the inverse medium scattering problem. In addition to the well-known capability in shape characterization, we demonstrate that the imaging indicator represents  nonlinear information about the unknown contrast. We further demonstrate additional representations for this nonlinear information  using the generalized linear sampling method. We shed light on this new result using the Born model and analytical examples.

\end{abstract}

\section{Introduction}
Inverse scattering merits important applications in non-destructive evaluation, seismic imaging, ocean acoustic and many others. The goal  is to retrieve information about the scattering object from measurement data. However, this is a challenging problem since inverse scattering is intrinsically ill-posed and nonlinear. To avoid incorrect a priori knowledge of the scattering object and to facilitate computational efficiency, the so-called qualitative methods have attracted much attention. 
In particular, the linear sampling method \cite{ColtonKirsch96} and the factorization method \cite{Kirsch98}  play   important roles in inverse problems associated with \textit{shape characterization} such as inverse scattering problem and electrical impedance tomography  \cite{kirsch2008factorization}. The idea of linear sampling and factorization methods is to build an imaging indicator $I(z)$ such that $I(z) < \infty$ if and only if $z$ is inside the support of the scattering object. For a more comprehensive introduction to qualitative methods, we refer to \cite{cakoni2016qualitative,colton2012inverse,kirsch2008factorization} and the references therein.

The work \cite{Kirsch17} investigates the factorization method with its connection to the Born model, which demonstrates the same ill-posedness of the factorization method and the Born factorization method. Since ill-posedness and nonlinearity are two fundamental characteristics of inverse scattering problems, we are hereby motivated to investigate the nonlinear nature in the linear sampling method while taking into account both its ill-posedness and the Born approximation model.

For the inverse Born scattering, it is well-known that the unknown contrast can be uniquely reconstructed via linear methods. Recently in the spirit of increasing stability \cite{IN2022}, we mention the work \cite{meng23data} which investigates a low-rank structure tailored for inverse Born scattering and  proves stability result for the unique reconstruction with $L^2$ perturbations using the generalized prolate spheroidal wave functions. This may indicate that the linear sampling method is capable of reconstructing the unknown (or its linear transformation) for linear inverse problems. Indeed for
a class of linear inverse problems (including the Born inverse scattering), the work \cite{audibertmeng23} shows that a particular formulation  of the linear sampling method allows not only shape characterization but also parameter characterization; loosely speaking, the indicator $I(z)$ of \cite{audibertmeng23} represents an average of $1/q$ when $z$ belongs to the support of the unknown parameter $q$. However, it remains interesting and challenging to quantify  the imaging indicator  in the nonlinear case.

Further insights may also be drawn from the dual space method \cite{{colton1988}} which solves a set of linear ill-posed problems (based on the far-field equation) and uses nonlinear optimization to find the unknown contrast. Certain combination of  far-field equations (with certain right hand sides) may lead to certain nonlinear information about the unknown contrast. This work is also motivated by kernel machine \cite{Herbrichbook,Hofmann2008} and deep learning \cite{Goodfellowbook} where nonlinear regression can be achieved by mapping the features to a high-dimensional space and by solving a linear regression in the new feature space.  The linear sampling method, if being understood in the spirit of nonlinear regression, may perform a linear regression in a high-dimensional space and thus may obtain   nonlinear information about the unknown using $I(z)$.

Our contribution is to quantify clearly the nonlinear information about the unknown contrast for the linear sampling method. 
We prove that the linear sampling indicator $I(z)$ represents a nonlinear information about the unknown contrast supported in $\Omega$.  In particular, with the factorization of the far-field operator $\mathcal{F} = \mathcal{H}^* \mathcal{T} \mathcal{H}$, we propose to find solutions $g_{z,\alpha}$ to the far-field equation using a family of regularization schemes $\mathcal{R}_\alpha$ with parameters $\alpha>0$; classical regularizations such as Tikhonov regularization, singular value cut off regularization, and Landweber iteration are examples in this regularization scheme.  The proposed imaging indicator $I(z)$ inherits the standard shape characterization that  $I(z) < \infty$ if and only if $z$ belongs to the support $\Omega$ of the scattering object, and this result has already been demonstrated in  \cite[Theorem 7.6]{kirsch2008factorization} using the Tikhonov regularization as an example. We also mention the similar  idea of \cite{arenslechleiter15} on inverse obstacle scattering. More importantly, we prove that the indicator  represents the following nonlinear information about the unknown contrast $q$,  
\begin{equation*} 
\lim_{\alpha \to 0} \langle  g_{z,\alpha}, \phi_z  \rangle_{L^2(\mathbb{S}^{d-1})} = \langle (\mathcal{P}_\Omega \mathcal{T} \mathcal{P}_\Omega)^{-1} E_z^{\mathcal{P}}  , E_z^{\mathcal{P}}   \rangle_{L^2(\Omega)},
\end{equation*}
where $\phi_z(\hat{x})=  e^{-ik\hat{x} \cdot z}$, $\mathcal{T}$  is a nonlinear operator determined by the contrast $q$ (cf. \eqref{Section operator T_omega def}),  $\mathcal{P}_\Omega$ is a projector operator depending only on $\Omega$ (cf. \eqref{Section operator def of Pomega}), and $E_z^{\mathcal{P}} = -(\Delta + k^2) w_z^{\mathcal{P}}$ where $w_z^{\mathcal{P}} \in H^2(\Omega)$ satisfies the fourth-order boundary value problem (cf. \eqref{Section FM and a LSM main of LSM theorem lemma PomegaEz fourthorder eqn1}--\eqref{Section FM and a LSM main of LSM theorem lemma PomegaEz fourthorder eqn3}).

There is a  connection between the new result and the generalized linear sampling method \cite{audiberthaddar15}; in particular,  we give an alternative proof of this new result and provide alternative characterizations for this nonlinear information.   As will be seen, the generalized linear sampling method need a specific regularization term which is less common than the regularization schemes in Section \ref{section subsection new indicator} and might be more difficult to cope with. However it is worth mentioning the more involved analysis of the generalized linear sampling method may allow to relax some hypothesis needed for the factorization method. Nevertheless, we demonstrate our result under standard assumptions on the unknown contrast.

The paper is further organized as follows. In Section \ref{Section Model}, we introduce the mathematical model for the inverse medium scattering problem and provide the necessary preliminaries for the linear sampling method; in particular we introduce a projection operator to help understand the factorization of the data operator. We investigate an alternative formulation of the linear sampling method in Section \ref{Section new LSM shape} and demonstrate its standard shape characterization for general regularization schemes. In Section \ref{Section parameter LSM} we prove how the imaging indicator quantifies the nonlinear information about the unknown contrast. The expression involves the unique projection $E_z^\mathcal{P}$ and demonstrate how to obtain it using a fourth-order boundary valued problem. With the help of the generalized linear sampling method, in Section \ref{Section GLSM interpretation} we provide alterative proof of the main result and demonstrate alternative representations for the nonlinear information about the unknown.
Finally in Section \ref{Section analysis inside}, we shed light on the imaging indicator in the Born scattering case and illustrate the unique projection $E_z^{\mathcal{P}}$ in the spherical and cylindrical symmetric cases. In the Appendix, we provide standard proofs of well-known results needed for a self-contained paper.

\section{Mathematical model for inverse medium scattering and preliminaries for the linear sampling method} \label{Section Model}
\subsection{Mathematical model}
In this section, we introduce the inverse medium scattering problem in $\mathbb{R}^d$, $d=2,3$.  Let $k>0$ be the  wave number. A plane wave  takes the following form 
\begin{equation*} %
e^{i k x \cdot \hat{\theta}}, \quad \hat{\theta} \in \mathbb{S}^{d-1} :=\{ x \in \mathbb{R}^d: |x|=1\},
\end{equation*}
where $\hat{\theta}$ is the direction of propagation.
Let $\Omega \subset \mathbb{R}^d$ be an open and bounded set with
Lipschitz boundary $\partial \Omega$ such that $\mathbb{R}^d \backslash \overline{\Omega}$  is connected. To best illustrate the main result, we let the real-valued function $q(x) \in L^\infty(\mathbb{R}^d)$ be the contrast of the medium, $q > 0$ on $\Omega$ and $q=0$ on $\mathbb{R}^d \backslash \overline{\Omega}$  such that the support of the contrast is $\Omega$. The contrast $q$ is related to physical quantities such as the electric permittivity and magnetic permeability for polarized electromagnetic scattering in two dimensions (cf. \cite{cakoni2016qualitative}) or refractive index in three dimensions (cf. \cite{colton2012inverse}). Throughout the paper we assume that  $q_{\rm sup} \ge q(y)\ge q_{\rm inf}$, a.e., $y \in \Omega$, for some positive constants $q_{\rm inf/sup}$.  The forward scattering problem is to find total wave filed $e^{i k x \cdot \hat{\theta} } + u^s(x;\hat{\theta};k)$ belonging to $H^1_{loc}(\mathbb{R}^d)$  such that
\begin{eqnarray}
\Delta_x \big( u^s(x;\hat{\theta};k) + e^{ikx\cdot \hat{\theta}} \big) + k^2 \left(1+q(x)\right) \big( u^s(x;\hat{\theta};k) + e^{ikx\cdot \hat{\theta}} \big) =  0 \quad &\mbox{in}& \quad \mathbb{R}^d, \quad  \label{medium us eqn1}\\
\lim_{r:=|x|\to \infty} r^{\frac{d-1}{2}}  \big( \frac{\partial u^s(x;\hat{\theta};k)}{\partial r} -ik u^s(x;\hat{\theta};k)\big) =0,  \label{medium us eqn2}
\end{eqnarray}
{where the last equation, i.e., the Sommerfeld radiation condition, holds  uniformly for all directions}. A solution is called   radiating   if it satisfies this radiation condition. The scattered wave field is $u^s(\cdot;\hat{\theta};k)$ due to the plane wave $e^{i k x \cdot \hat{\theta} }$.   The above scattering problem \eqref{medium us eqn1}--\eqref{medium us eqn2} is a special case of the more general problem where one looks for a radiating solution $u^s \in H^1_{loc}(\mathbb{R}^d)$   to
\begin{equation} \label{medium us eqn1+2}
\Delta u^s + k^2(1+q) u^s = -k^2 q f,
\end{equation}
where $f \in L^\infty(\mathbb{R}^d)$. Setting $f=e^{ikx\cdot \hat{\theta}}$ in \eqref{medium us eqn1+2} gives \eqref{medium us eqn1}--\eqref{medium us eqn2}. This model is referred to as the full model.   

The fundamental solution for the Helmholtz equation is given by
\begin{equation} \label{Green function def}
\Phi(x,y)
:=
\left\{
\begin{array}{cc}
\frac{i}{4} H^{(1)}_0(k|x-y|)  & d=2        \\
\frac{e^{ik |x-y|}}{4\pi|x-y|} & d=3
\end{array}
\right. \qquad x\not=y,
\end{equation}
where $H^{(1)}_0$ denotes the Hankel function of the first kind (\cite{colton2012inverse}). We state the following lemma where we refer to \cite{colton2012inverse,kirsch2008factorization} and \cite{Kirsch17} for a complete proof.
\begin{lemma}
There exists a unique radiating solution to \eqref{medium us eqn1+2}. The solution can be solved with the help of  the Lippmann-Schwinger integral equation 
\begin{eqnarray*}
u^s(x) - k^2 \int_\Omega \Phi(x,y)q(y)  u^s(y)  \ind y= k^2 \int_\Omega \Phi(x,y)q(y)   f(y) \ind y, \quad x \in \mathbb{R}^d.
\end{eqnarray*}
\end{lemma}

Note that (c.f. \cite{cakoni2016qualitative})
\begin{equation*} 
u^{s}(x;\hat{\theta};k)
=
\left\{
\begin{array}{cc}
\frac{e^{i\frac{\pi}{4}}}{\sqrt{8k\pi}} \frac{e^{ikr}}{\sqrt{r}}\left\{u^{\infty}(\hat{x};\hat{\theta};k)+\mathcal{O}\left(\frac{1}{r}\right)\right\}  & d=2        \\
 \frac{e^{ik r}}{4\pi r} \left\{ u^\infty (\hat{x};\hat{\theta};k) +     \mathcal{O}\left(\frac{1}{r}\right) \right\}& d=3
\end{array}
\right.
 \quad\mbox{as }\,r=|x|\rightarrow\infty,
\end{equation*}
uniformly with respect to all directions $\hat{x}:=x/|x|\in\mathbb{S}^{d-1}$, we arrive at $u^{\infty}(\hat{x};\hat{\theta};k)$ which is known as the  far-field pattern with $\hat{x}\in\mathbb{S}^{d-1}$ denoting the observation direction. Here $\mathbb{S}^{d-1} = \{x \in \mathbb{R}^d: |x|=1\}$ denotes the unit circle in two dimensions or the unit sphere in three dimensions. The multi-static data at a fixed frequency are given by
\begin{equation} \label{Section model far-field data}
\{u^{\infty}(\hat{x};\hat{\theta};k): \hat{x}\in\mathbb{S}^{d-1}, \hat{\theta}\in\mathbb{S}^{d-1}\}.
\end{equation}

In this work, the inverse scattering problem is to retrieve information about the contrast $q$ from the multi-static data \eqref{Section model far-field data}. We are interested in the  linear sampling method and aim to characterize the nonlinear information in addition to the standard shape characterization.
\subsection{Preliminaries for the linear sampling method}\label{Section operator}
To begin with the analysis, we introduce the necessary preliminaries for the qualitative methods in order to analyze the inverse scattering problem.  Let the far-field operator $\mathcal{F}: L^2(\mathbb{S}^{d-1}) \to L^2(\mathbb{S}^{d-1})$ be given by
\begin{equation} \label{Section operator N def}
\left( \mathcal{F} g \right)(\hat{x}) := \int_{\mathbb{S}^{d-1}}  u^{\infty}(\hat{x};\hat{\theta};k) g(\hat{\theta}) \ind s(\hat{\theta}),
\end{equation}
where the kernel is given by the data \eqref{Section model far-field data}. 

It is known that the far-field operator has a factorization as follows. Introduce $\mathcal{H}: L^2(\mathbb{S}^{d-1}) \to L^2(\Omega)$ by 
\begin{equation} \label{Section operator S_omega def}
\left( \mathcal{H} g \right)(y) := \int_{\mathbb{S}^{d-1}} e^{i k y \cdot \hat{\theta}} g(\hat{\theta}) \ind s(\hat{\theta}), \quad \forall g \in L^2(\mathbb{S}^{d-1}), \quad y \in \Omega,
\end{equation}
and it follows directly that its adjoint $\mathcal{H}^*: L^2(\Omega) \to L^2(\mathbb{S}^{d-1})$ is given by
\begin{equation} \label{Section operator S*_omega def}
\left( \mathcal{H}^* h \right)(\hat{x}) := \int_{\Omega} e^{-i k \hat{x} \cdot y} h(y) \ind y, \quad \forall h \in L^2(\Omega), \quad \hat{x} \in \mathbb{S}^{d-1}
\end{equation}
which is dictated by $\langle \mathcal{H}^* h,g \rangle_{L^2(\mathbb{S}^{d-1})} =  \langle  h, \mathcal{H} g \rangle_{L^2(\Omega)}$. Here $\langle , \rangle_{L^2(\Omega)}$ represents the $L^2(\Omega)$ inner product with conjugation in the second argument, and we further denote by $\|\cdot\|_{L^2(D)}$ the corresponding $L^2(D)$ norm. From now on we drop the subscript $L^2(\mathbb{S}^{d-1})$ when the inner product is in $L^2(\mathbb{S}^{d-1})$ and will explicitly indicate a subscript for other cases. Another operator $\mathcal{T}$ is needed for the factorization, namely $\mathcal{T}: L^2(\Omega) \to L^2(\Omega)$ which is given by 
\begin{equation} \label{Section operator T_omega def}
 \mathcal{T} f   := k^2q f+ k^2 q v|_\Omega, \quad \forall f \in L^2(\Omega),
\end{equation} 
where $v\in H^1_{loc}(\mathbb{R}^d)$ is the unique radiating solution to 
\begin{equation} \label{Section operator T_omega def v def}
\Delta v + k^2(1+q) v = -k^2 q f.
\end{equation}

The following Theorem gives the factorization of the far-field operator. The proof is standard and we include one in  Appendix \ref{section appendix factorization} for completeness.
\begin{lemma} \label{Section operator Omega fac theorem}
Let the data operator $\mathcal{F}: L^2(\mathbb{S}^{d-1}) \to L^2(\mathbb{S}^{d-1})$ be given by \eqref{Section operator N def}. Then it holds that
\begin{equation*}
\mathcal{F} = \mathcal{H}^* \mathcal{T} \mathcal{H}
\end{equation*}
where $\mathcal{H}$, $\mathcal{H}^*$, and $\mathcal{T}$ are given by \eqref{Section operator S_omega def}, \eqref{Section operator S*_omega def}, and \eqref{Section operator T_omega def}, respectively. 
\end{lemma}

It is well-known that the Herglotz wave functions are of the form
$$
\int_{\mathbb{S}^{d-1}} e^{i k x \cdot \hat{\theta}} g(\hat{\theta}) \ind s(\hat{\theta}), \quad \forall x\in\mathbb{R}^d,
$$
and every function in $R(\mathcal{H})$ is the restriction of Herglotz wave functions in $\Omega$; here $R(\mathcal{H})$ denotes the range of $\mathcal{H}$.
For later purposes, let $Y_\Omega$ be the closure of the range of $\mathcal{H}$, i.e., $Y_\Omega:=\overline{ R(\mathcal{H})}$. Note that $R(\mathcal{H})$ is dense in $\{v \in L^2(\Omega): \Delta v + k^2 v =0 \mbox{ in } \Omega \mbox{ variationally}\}$(c.f.  \cite[Lemma 6.45]{cakoni2016qualitative} and \cite[Theorem 7.3]{kirsch2008factorization}), then $Y_\Omega=\{v \in L^2(\Omega): \Delta v + k^2 v =0 \mbox{ in } \Omega \mbox{ variationally}\}$.  Define the projection operator
\begin{equation}\label{Section operator def of Pomega}
\mathcal{P}_\Omega: L^2(\Omega) \to L^2(\Omega)  \mbox{ where } \mathcal{P}_\Omega w = w|_{Y_\Omega},\quad \forall w \in L^2(\Omega).   
\end{equation}
Then it follows that $\mathcal{H}^* (\mathcal{P}_\Omega-\mathcal{I}) = 0$ (where $\mathcal{I}$ is the identity operator) since $Y_\Omega$ coincides with the orthogonal complement of the null space of $\mathcal{H}^*$. As a consequence, we have that
\begin{equation}  \label{Section operator Omega fac coercive factorization}
\mathcal{F} = \mathcal{H}^* \mathcal{P}_\Omega \mathcal{T} \mathcal{P}_\Omega \mathcal{H}.
\end{equation}
This equivalent formulation is motivated by that the middle operator $\mathcal{P}_\Omega \mathcal{T} \mathcal{P}_\Omega$ is  coercive provided $k$ is not  an interior transmission eigenvalue. To be more precise, we first introduce the definition of the interior transmission eigenvalue following \cite{cakoni2016inverse} and we refer to \cite{cakoni2016inverse,Cakoni2010,Kirsch21}  for a more comprehensive discussion.

\begin{definition}\label{Section operator ITE def}
$k \in \mathbb{C}$ is called an {\it interior transmission eigenvalue} if there exists a non-trivial pair $(w,v) \in L^2(\Omega) \times L^2(\Omega)$ such that $w-v \in H_0^2(\Omega)$ and
\begin{eqnarray*}
\Delta w + k^2(1+q) w =0 &\mbox{ in }& \Omega,\\
\Delta v + k^2v =0 &\mbox{ in }& \Omega,\\
w = v &\mbox{ on }& \partial \Omega,\\
\frac{\partial w}{\partial \nu} = \frac{\partial v}{\partial \nu} &\mbox{ on }& \partial \Omega.
\end{eqnarray*}
Here $H_0^2(\Omega):=\{ u\in H^2(\Omega): u=0 \mbox{ on } \partial \Omega \mbox{ and } \frac{\partial u}{\partial \nu} \mbox{ on } \partial \Omega\}$.
\end{definition}
Now we are ready to study the middle operator $\mathcal{P}_\Omega \mathcal{T} \mathcal{P}_\Omega$ in Proposition \ref{Section operator middle positive definite prop}, whose proof is standard and we include one in Appendix \ref{Section appendix proof of middle operator} for completeness.
\begin{proposition}\label{Section operator middle positive definite prop}
Assume that $q_{\rm sup} \ge q(y)\ge q_{\rm inf}$, a.e., $y \in \Omega$, for some positive constants $q_{\rm inf/sup}$.  Then we have the following properties.
\begin{enumerate}
\item $  \mathcal{T}  = \mathcal{T}_{b} + \mathcal{C}$ where $\mathcal{T}_{b}: L^2(\Omega) \to L^2(\Omega)$ is given by 
\begin{equation} \label{Section operator T_omega def Born}
\mathcal{T}_{b} f   := k^2q f, \quad \forall f \in L^2(\Omega),
\end{equation}  
and $\mathcal{C}: L^2(\Omega) \to L^2(\Omega)$ is given by
$
\mathcal{C}  :=  \mathcal{T}-\mathcal{T}_{b}$. 
Here $\mathcal{T}_{b}$ is coercive and $\mathcal{C} $ is compact.  
\item $\Im \mathcal{T}$ is non-negative, i.e.,
$$\Im \big( \langle \mathcal{T} f,f\rangle_{L^2(\Omega)}\big) \ge 0, \quad \forall f \in L^2(\Omega).
$$
 \item Moreover assume that $k$ is not an interior transmission eigenvalue, then
 $\Im \mathcal{T}$ is positive on $Y_\Omega$, i.e.,
$$\Im \big( \langle \mathcal{T} f,f\rangle_{L^2(\Omega)}\big) > 0, \quad \forall f \in Y_\Omega.
$$
\item 
Assume that $k$ is not an interior transmission eigenvalue, then $\mathcal{P}_\Omega \mathcal{T} \mathcal{P}_\Omega: Y_\Omega \to Y_\Omega$ is coercive and there exists a positive constant $T_{\inf} >0$ such that
\begin{equation} \label{Section operator middle positive definite prop eqn}
|\langle \mathcal{P}_\Omega \mathcal{T} \mathcal{P}_\Omega f,f\rangle_{L^2(\Omega)}| \ge T_{\inf} \| f\|_{L^2(\Omega)}^2, \quad \forall 0\not=f \in Y_\Omega.
\end{equation}
\end{enumerate}

 \end{proposition}
\begin{remark}
Proposition \ref{Section operator middle positive definite prop} holds also when $q_{\rm inf/sup}$ are some  negative constants by directly applying the same proof with $-\mathcal{T}_{0}$ being coercive. It is also possible to consider the case when $q$ changes sign strictly inside $\Omega$ (or equivalently when $q$ has a constant sign in a neighborhood of $\partial \Omega$), cf. \cite{Audibertsign,ArminLak}.
\end{remark}

It is well-known that the far-field operator is normal (see \cite[Remark 7.33]{cakoni2016qualitative} for $d=2$ and \cite[Theorem 4.4]{kirsch2008factorization} for $d=3$; see also Appendix \ref{Section appendix F is normal} for a self-contained paper) for real-valued contrast, then
 we can introduce the eigensystem $\{ \zeta_n,\mu_n\}_{n=0}^\infty$ of the normal operator $\mathcal{F}$ by
\begin{equation} \label{Section operator Omega eigensystem}
\mathcal{F} \zeta_n = \mu_n \zeta_n, \quad n=0,1,\cdots,
\end{equation}
here $\zeta_n \in L^2(\mathbb{S}^{d-1})$ and $\mu_n \in \mathbb{C}$. With such an eigensystem, one can define $|\mathcal{F}|^{1/2}: L^2(\mathbb{S}^{d-1}) \to L^2(\mathbb{S}^{d-1})$ by
\begin{equation*}
|\mathcal{F}|^{1/2} \zeta_n = |\mu_n|^{1/2} \zeta_n, \quad n=0,1,\cdots.
\end{equation*}
In addition, we have that the phases of $\frac{\mu_n}{|\mu_n|}$ are always in an interval of length strictly less than $\pi$ if $k$ is not an interior transmission eigenvalue. In particular the following is from \cite[pp. 190]{cakoni2016qualitative} and for a self-contained paper we include a brief proof in Appendix \ref{Section appendix proof of accumulation interval}.
 \begin{lemma}  \label{Section operator mu/|mu| phase interval}
Assume that $k$ is not an interior transmission eigenvalue, then $\Im \mu_n > 0$  and the accumulation point of $\frac{\mu_n}{|\mu_n|}$ cannot be $-1$ whereby $\frac{\mu_n}{|\mu_n|} = e^{i \eta_n}$ with $\eta_n \in [0,\pi-2\eta_\delta)$ with some $\eta_\delta \in (0,\pi/2]$. 
 \end{lemma}

We now state the following lemma on range identity.

\begin{lemma}\label{Section operator Omega range characterization prop}
Assume that $k$ is not an interior transmission eigenvalue. Then it follows that
$Range(\mathcal{H}^*) = Range(|\mathcal{F}|^{1/2})$.
\end{lemma}
\begin{proof}
Since $k$ is not an interior transmission eigenvalue, then the middle operator $\mathcal{P}_\Omega\mathcal{T} \mathcal{P}_\Omega$ is coercive, then the proof follows from \cite[Corollary 1.22]{kirsch2008factorization} and Proposition \ref{Section operator middle positive definite prop}.
\end{proof}

\section{Imaging indicator for the alternative linear sampling method} \label{Section new LSM shape}
In this section we study an alternative formulation of the linear sampling method (i.e., alternative linear sampling method). To begin with, let $\phi_z \in L^2(\mathbb{S}^{d-1})$ be given by 
\begin{equation*} \label{Section FM and a LSM phi_z def}
\phi_z(\hat{x}) :=  e^{-ik\hat{x} \cdot z},\quad\hat{x}\in \mathbb{S}^{d-1}
\end{equation*}

We first give the following well-known result from \cite[Theorem 4.6]{kirsch2008factorization} and we include a proof in Appendix \ref{Section appendix FM and a LSM lemma phiz S Omega} for a self-contained paper.
\begin{lemma} \label{Section FM and a LSM lemma phiz S Omega}
It holds that $z\in \Omega $ if, and only if, $\phi_z \in \mbox{Range} (\mathcal{H}^*)$. Moreover let $B(z,\epsilon):=\{x\in \mathbb{R}^d: |x-z| < \epsilon\}$ for some fixed small $\epsilon>0$. If $\overline{B(z,\epsilon)} \subset \Omega$, then it holds that
$$
\phi_z = \mathcal{H}^* E_z,
$$
where $E_z \in L^2(\Omega)$ is given by
\begin{equation*} \label{Section FM and a LSM E_z def}
E_z := \left\{
\begin{array}{cc}
  -(\Delta w_z + k^2 w_z ) & \mbox{in } B(z,\epsilon)   \\
0  &   \mbox{otherwise}
\end{array}
\right.,
\end{equation*}
here $w_z(x):= \chi_{|x-z|}\Phi(x,z)$ in $\mathbb{R}^d$, and $\chi \in C^\infty(\mathbb{R})$ is a cut off function with $\chi(t)=1$ for $|t| \ge \epsilon$ and $\chi(t)=0$ for $|t| \le \epsilon/2$, and $G$ is the fundamental solution given by \eqref{Green function def}.
\end{lemma}

The linear sampling method (LSM) and factorization method (FM) for shape characterization are as follows.
\begin{itemize}
\item[(LSM)] The linear sampling method solves the data equation $\mathcal{F} g_z \approx \phi_z$
using a regularization scheme to get a regularized solution $g_{z,\alpha}$ and indicates that
$
\|g_{z,\alpha}\|_{L^2(\mathbb{S}^{d-1})}
$
is large for   $z \in B\backslash \overline{\Omega}$ and is bounded for $z$ with $z \in  \Omega$ (due to Proposition \ref{Section operator Omega range characterization prop} and Lemma \ref{Section FM and a LSM lemma phiz S Omega}). This  is suggested by a partial theory similar to \cite{ColtonKirsch96}; we omit this partial theory since we will show a formulation of the linear sampling method with complete theoretical justification later on.
\item[(FM)] A direct application of Lemma \ref{Section operator Omega range characterization prop}
 and Lemma \ref{Section FM and a LSM lemma phiz S Omega} yields the factorization method: Assume that $k$ is not an interior transmission eigenvalue. Then $z\in \Omega $ if, and only if, $\phi_z \in \mbox{Range}(|\mathcal{F}|^{1/2})$. Here
\begin{equation}\label{Section FM and a LSM FM main result}
\phi_z \in Range(|\mathcal{F}|^{1/2}) \Longleftrightarrow \sum_{n=0}^\infty \frac{|\langle \phi_z, \zeta_n \rangle|^2}{|\mu_n|} < \infty,
\end{equation}
where  $\{ \zeta_n,\mu_n\}_{n=0}^\infty$ is the eigensystem of  $\mathcal{F}$ given by \eqref{Section operator Omega eigensystem}.
\end{itemize}

\subsection{Imaging indicator} \label{section subsection new indicator}
In the following, we study the alternative linear sampling method in the form of 
$$
\langle  g_{z,\alpha}, \phi_z  \rangle_{L^2(\mathbb{S}^{d-1})}
$$
and we demonstrate later a new interpretation of this indicator.    In this section we first demonstrate its viability in shape characterization. The idea is similar to the earlier work \cite{arenslechleiter15,audiberthaddar15} in inverse scattering to justify or generalize the linear sampling method. Such an alternative formulation $\langle  g_{z,\alpha}, \phi_z  \rangle_{L^2(\mathbb{S}^{d-1})}$ may be dated back  to the first paper on linear sampling method of Colton and Kirsch  \cite{ColtonKirsch96}. Indeed this particular formulation  $\langle  g_{z,\alpha}, \phi_z  \rangle_{L^2(\mathbb{S}^{d-1})}$ has already been used in \cite[Theorem 7.6]{kirsch2008factorization} to demonstrate the shape characterization using the Tikhonov regularization, see also \cite{arenslechleiter15} for inverse obstacle scattering. In the following, we show the shape characterization using general regularization schemes.

To begin with, we introduce a family of regularization schemes $\{ \mathcal{R}_\alpha\}_{\alpha>0}$ by
\begin{equation} \label{Section FM and a LSM R_alpha def}
\mathcal{R}_\alpha h := \sum_{n=0}^\infty f_\alpha(|\mu_n|^2) \overline{\mu_n} \langle h, \zeta_n \rangle \zeta_n,
\end{equation}
where $f_\alpha$ is a regularizing filter that is a bounded, real-valued, and piecewise continuous function $f_\alpha: (0, \infty) \to \mathbb{R}$ such that
\begin{equation} \label{Section FM and a LSM f_alpha prop}
\lim_{\alpha \to 0 } f_\alpha(\mu) = \frac{1}{\mu} \mbox{ for all } \mu>0,  \quad |\mu f_\alpha(\mu)| \le d_0 \mbox{ for all } \alpha \ge 0 \mbox{ and } \mu>0,
\end{equation}
here $d_0>0$ is a constant. 

With this family of  regularization schemes $\{ \mathcal{R}_\alpha\}_{\alpha>0}$, one can introduce a family of regularized solutions by $g_{z,\alpha} = \mathcal{R}_\alpha \phi_z$.
Classical  regularizations include the Tikhonov regularization with 
$$
f_\alpha(\mu) \to \frac{1}{\mu+\alpha} \quad\mbox{ so that }\quad
g_{z,\alpha} \to  \sum_{n=0}^\infty \frac{ \overline{\mu_n} }{\mu_n^2 + \alpha} \langle \phi_z, \zeta_n \rangle \zeta_n,
$$
and the singular value cut off regularization with
$$
f_\alpha(\mu) \to  \left\{
\begin{array}{cc}
1/\mu  & \mu \ge \alpha    \\
0  &   \mbox{otherwise}
\end{array}
\right.
\quad \mbox{ so that }\quad
g_{z,\alpha} \to  \sum_{\mu_n \ge \alpha}^\infty \frac{1}{\mu_n} \langle \phi_z, \zeta_n \rangle \zeta_n.
$$
The Landerweber iterative method \cite{Kirsch21} also belongs to this regularization family.
Similar to the inverse obstacle scattering case \cite{arenslechleiter15}, our shape characterization result is as follows. We point out that \cite[Theorem 7.6]{kirsch2008factorization} already demonstrated the following result using the Tikhonov regularization.
\begin{theorem}\label{Section FM and a LSM main of LSM theorem}
Assume that $k$ is not an interior transmission eigenvalue.
Suppose that $\{ \mathcal{R}_\alpha\}_{\alpha>0}$ is a family of regularization schemes given by \eqref{Section FM and a LSM R_alpha def}-- \eqref{Section FM and a LSM f_alpha prop} and set $g_{z,\alpha}  = \mathcal{R}_\alpha \phi_z$. The   following characterizations of the support $\Omega$ hold.
\begin{itemize}
\item If $z \in B\backslash \overline{\Omega}$, then $\langle  g_{z,\alpha}, \phi_z  \rangle_{L^2(\mathbb{S}^{d-1})}$ cannot remain bounded as $\alpha \to 0$. 
\item If $z \in \Omega$, then $\langle  g_{z,\alpha}, \phi_z  \rangle_{L^2(\mathbb{S}^{d-1})}$ remains bounded as $\alpha \to 0$. 
\end{itemize}

\end{theorem}
\begin{proof}
We first derive an expression of $\langle  g_{z,\alpha}, \phi_z  \rangle_{L^2(\mathbb{S}^{d-1})}$. From the definition of $g_{z,\alpha}$, one gets  $g_{z,\alpha}  = \mathcal{R}_\alpha \phi_z$ whereby
\begin{equation*}  
g_{z,\alpha} =  \sum_{n=0}^\infty f_\alpha(\mu_n^2) \overline{\mu_n} \langle \phi_z, \zeta_n \rangle \zeta_n,
\end{equation*}
in this way we obtain
\begin{eqnarray} \label{Section FM and a LSM main of LSM theorem proof expression}
&&\langle g_{z,\alpha},  \phi_z \rangle = \sum_{n=0}^\infty f_\alpha(\mu_n^2) \overline{\mu_n}  |\langle \phi_z, \zeta_n \rangle|^2 = \sum_{n=0}^\infty f_\alpha(\mu_n^2)|\mu_n|^2   \frac{\overline{\mu_n}}{|\mu_n|}  \frac{1}{|\mu_n|}  |\langle \phi_z, \zeta_n \rangle|^2.
\end{eqnarray}

When $z \in B\backslash \overline{\Omega}$, first note from the factorization method result \eqref{Section FM and a LSM FM main result} that $\phi_z \not\in Range(|\mathcal{F}|^{1/2})$ so that 
\begin{equation}\label{Section FM and a LSM main of LSM theorem proof FM diverge}
\sum_{n=0}^\infty \frac{|\langle \phi_z, \zeta_n \rangle|^2}{|\mu_n|} = \infty.
\end{equation}
From Lemma \ref{Section operator mu/|mu| phase interval} where $\frac{\mu_n}{|\mu_n|} = e^{i \eta_n}$ with $\eta_n \in [0,\pi-2\eta_\delta)$ and $\eta_\delta \in (0,\pi/2]$, then we have
 \begin{eqnarray*}
\Im (-e^{-i\eta_\delta } \langle g_{z,\alpha},  \phi_z \rangle ) &=& \sum_{n=0}^\infty f_\alpha(\mu_n^2)|\mu_n|^2     \Im(-e^{-i\eta_\delta }e^{-i \eta_n})  \frac{1}{|\mu_n|}  |\langle \phi_z, \zeta_n \rangle|^2 \\
&\ge& \sum_{n=0}^\infty f_\alpha(\mu_n^2)|\mu_n|^2     \sin(\eta_\delta) \frac{1}{|\mu_n|}  |\langle \phi_z, \zeta_n \rangle|^2  .
\end{eqnarray*}
Now since the series \eqref{Section FM and a LSM main of LSM theorem proof FM diverge} diverges, then for any large $M>0$, there exists $N_M>0$ such that 
$$
 \sum_{n=0}^{N_M} \frac{|\langle \phi_z, \zeta_n \rangle|^2}{|\mu_n|} >2M,
$$
then one can chose $\alpha_M>0$ (due to the property of $f_\alpha$ in  \eqref{Section FM and a LSM f_alpha prop}) such that
$$
f_\alpha(\mu_n^2) > \frac{1}{2\mu_n^2}, \quad \forall \alpha \in (0,\alpha_M),\quad \mbox{for all }  n=0,1,\cdots,N_M.
$$
This yields that
$$
\Im (-e^{-i\eta_\delta } \langle g_{z,\alpha},  \phi_z \rangle )   \ge \sin(\eta_\delta)\sum_{n=0}^{N_M} f_\alpha(\mu_n^2) |\mu_n| |\langle \phi_z, \zeta_n \rangle|^2 > \frac{\sin(\eta_\delta)}{2}\sum_{n=0}^{N_M} \frac{|\langle \phi_z, \zeta_n \rangle|^2}{\mu_n} >\sin(\eta_\delta) M,
$$
for all $\alpha \in (0,\alpha_M)$.
This proves $\lim_{\alpha \to 0}\Im (-e^{-i\eta_\delta } \langle g_{z,\alpha},  \phi_z \rangle ) = \infty$, i.e.,  $\langle  g_{z,\alpha}, \phi_z  \rangle_{L^2(\mathbb{S}^{d-1})}$ cannot remain bounded as $\alpha \to 0$. 

Now we consider the case when $z \in \Omega$, then from the factorization method result \eqref{Section FM and a LSM FM main result} one can obtain that there  exists the unique solution $g_z^{FM} \in L^2(B)$ to $|\mathcal{F}|^{1/2} g_z^{FM} = \phi_z$   and
$$
 \|g_z^{FM}\|^2 = \sum_{n=0}^\infty \frac{|\langle \phi_z, \zeta_n \rangle|^2}{|\mu_n|} < \infty.
$$
Note that $f_\alpha$ satisfies \eqref{Section FM and a LSM f_alpha prop}, then we have from \eqref{Section FM and a LSM main of LSM theorem proof expression} that
$$
|\langle  g_{z,\alpha}, \phi_z  \rangle_{L^2(\mathbb{S}^{d-1})}|  \le \sum_{n=0}^\infty f_\alpha(\mu_n^2) |\mu_n|^2  \frac{1}{|\mu_n|} |\langle \phi_z, \zeta_n \rangle|^2  \le d_0  \sum_{n=0}^\infty \frac{|\langle \phi_z, \zeta_n \rangle|^2}{\mu_n} < \infty.
$$
i.e., $\langle  g_{z,\alpha}, \phi_z  \rangle_{L^2(\mathbb{S}^{d-1})}$ remains bounded as $\alpha \to 0$. 
This  completes the proof.
\end{proof}
The indicator function $\langle  g_{z,\alpha}, \phi_z  \rangle_{L^2(\mathbb{S}^{d-1})}$ in Theorem \ref{Section FM and a LSM main of LSM theorem} allows to determine   the support  $\Omega$. In \cite{audibertmeng23} the authors showed that this indicator allows to reconstruct the average of the unknown for the linearized problem. For the nonlinear case, we show in the next section -- through the lens of the inverse scattering problem -- that this indicator $\langle  g_{z,\alpha}, \phi_z  \rangle_{L^2(\mathbb{S}^{d-1})}$ represents a weighted average of a nonlinear information about the unknown.

\section{Characterizing nonlinear information for the imaging indicator} \label{Section parameter LSM}
To demonstrate a new interpretation of the indicator introduced in Section \ref{section subsection new indicator}, 
We first prove the following lemmas that are needed for the main theorem.
\begin{lemma}\label{Section parameter LSM lemma}
Assume that $k$ is not an interior transmission eigenvalue.
Suppose that $\{ \mathcal{R}_\alpha\}_{\alpha>0}$ is a family of regularization schemes given by \eqref{Section FM and a LSM R_alpha def}-- \eqref{Section FM and a LSM f_alpha prop} and set $g_{z,\alpha}  = \mathcal{R}_\alpha \phi_z$. Then it holds for any   $B(z,\epsilon) \subset \Omega$ that 
\begin{equation*}%
\| \mathcal{H} g_{z,\alpha}\|_{L^2(\Omega)} \le \frac{d_0}{T_{inf}} \|E_z\|_{L^2(\Omega)}, \quad \forall \alpha>0,
\end{equation*}
where $d_0$ given by \eqref{Section FM and a LSM f_alpha prop} is a constant independent of $\alpha$, and $T_{inf}>0$ is the lower bound for the operator $\mathcal{T}$ given by \eqref{Section operator middle positive definite prop eqn}.
\end{lemma}
\begin{proof}
Note from Proposition \ref{Section operator middle positive definite prop} that  $\mathcal{P}_\Omega\mathcal{T}\mathcal{P}_\Omega:Y_\Omega \to Y_\Omega$ is coercive and that $\mathcal{H}  g_{z,\alpha} \in Y_\Omega$, then it follows that  
\begin{eqnarray}\label{Section parameter LSM lemma proof eqn1}
&& T_{\inf} \|  \mathcal{H}  g_{z,\alpha} \|^2_{L^2(\Omega)} \le  |\langle \mathcal{P}_\Omega \mathcal{T} \mathcal{P}_\Omega \mathcal{H}  g_{z,\alpha},   \mathcal{H}  g_{z,\alpha}\rangle_{L^2(\Omega)}|=| \langle \mathcal{F}  g_{z,\alpha},   g_{z,\alpha} \rangle|.\end{eqnarray}
Note that $g_{z,\alpha}$ is given by  $g_{z,\alpha}  = \mathcal{R}_\alpha \phi_z = \sum_{n=0}^\infty f_\alpha(\mu_n^2) \overline{\mu_n} \langle \phi_z, \zeta_n \rangle \zeta_n$ whereby
\begin{equation*}  
\mathcal{F}g_{z,\alpha} =  \sum_{n=0}^\infty f_\alpha(\mu_n^2) |\mu^2_n| \langle \phi_z, \zeta_n \rangle \zeta_n,
\end{equation*}
which yields (where one notes that $f_\alpha$ and $\mu_n$ are real-valued)
\begin{equation*}  
|\langle \mathcal{F}  g_{z,\alpha},   g_{z,\alpha} \rangle| =  \sum_{n=0}^\infty [f_\alpha(\mu_n^2)]^2 |\mu_n|^3 |\langle \phi_z, \zeta_n \rangle|^2 \overset{\eqref{Section FM and a LSM f_alpha prop}}{\le}  d_0\sum_{n=0}^\infty f_\alpha(\mu_n^2) |\mu_n| |\langle \phi_z, \zeta_n \rangle|^2 = d_0 |\langle \phi_z,   g_{z,\alpha} \rangle|,
\end{equation*}
where it is noted that $|\langle \phi_z,   g_{z,\alpha} \rangle|$ is always bounded since $z \in \Omega$ and Theorem~\ref{Section FM and a LSM main of LSM theorem}.

Now from the above equation and \eqref{Section parameter LSM lemma proof eqn1},  we have that $ \|\mathcal{P}_\Omega\mathcal{H}  g_{z,\alpha} \|^2_{L^2(\Omega)} <\infty$ and
\begin{eqnarray*}
&&T_{\inf} \| \mathcal{H} g_{z,\alpha}\|^2_{L^2(\Omega)} \le \langle |\mathcal{F}  g_{z,\alpha},   g_{z,\alpha}| \rangle \le  d_0 |\langle  \phi_z,    g_{z,\alpha}  \rangle| = d_0 |\langle E_z,   \mathcal{H} g_{z,\alpha} \rangle_{L^2(\Omega)}| \\
&\le& d_0 \|E_z\|_{L^2(\Omega)}    \|   \mathcal{H} g_{z,\alpha}\|_{L^2(\Omega)},
\end{eqnarray*}
this proves the lemma.
\end{proof}
The main tool of \cite{audibertmeng23} is the prolate spheroidal wave functions and their generalizations which diagonalize the (left and right) factorized operators and sometimes diagonalize the data operator. In the inverse medium scattering case, this is no longer possible since the operator $\mathcal{H}$ is an operator from $L^2(\mathbb{S}^{d-1})$ to $L^2(\Omega)$. One possible idea may be to use a Riesz basis (which was used in the original paper on the factorization method \cite{Kirsch98}) constructed from the eigenfunction $\zeta_n$  \eqref{Section operator Omega eigensystem} of the data operator. In the following we give another set of basis in $Y_\Omega$ which plays a similar role of the prolate spheroidal wave functions and their generalizations in \cite{audibertmeng23}.

\begin{lemma} \label{Section parameter LSM lemma PH*HP}
The operator $\mathcal{P}_\Omega\mathcal{H} \mathcal{H}^* \mathcal{P}_\Omega: Y_\Omega \to Y_\Omega$ is positive definite, compact, and self-adjoint. There exists an orthonormal system $\{ \psi_j \}_{j=0}^\infty$ in $Y_\Omega$ (where we equip $Y_\Omega$ with $L^2(\Omega)$ norm) with corresponding non-zero eigenvalues $\{ |\lambda_j|^2 \}_{j=0}^\infty$ in non-increasing order such that
\begin{equation}\label{Section parameter LSM lemma PH*HP eigensystem}
\mathcal{P}_\Omega\mathcal{H} \mathcal{H}^* \mathcal{P}_\Omega \psi_j = |\lambda_j|^2 \psi_j, \quad \forall j=0,1,\cdots.
\end{equation}

\end{lemma}
\begin{proof}
First note that $\mathcal{P}_\Omega=\mathcal{P}_\Omega^*$, then $\mathcal{P}_\Omega\mathcal{H} \mathcal{H}^* \mathcal{P}_\Omega$ is self-adjoint and compact; furthermore for any $\psi \in L^2(\Omega)$, 
$$
<\mathcal{P}_\Omega\mathcal{H} \mathcal{H}^* \mathcal{P}_\Omega \psi, \psi>_{L^2(\Omega)} = < \mathcal{H}^* \mathcal{P}_\Omega \psi, \mathcal{H}^* \mathcal{P}_\Omega\psi>_{L^2(\Omega)} \ge 0,
$$
so that all the eigenvalues are non-negative and if there is a zero eigenvalue, then $\mathcal{H}^* \mathcal{P}_\Omega$ is not injective. However this contradicts to the fact that $\mathcal{H}^*\mathcal{P}_\Omega$ is injective on $Y_\Omega$ due to that $\mathcal{P}_\Omega$ is the projection onto $Y_\Omega$ which is the closure of $\mbox{Range} (\mathcal{H})$ (which coincides with the orthogonal complement of the null space of $\mathcal{H}^*$). Then an application of the spectral theory (see for instance \cite[pp. 346--347]{Kirsch21})  proves the lemma.
\end{proof}

Now we are ready to prove the following new interpretation of the imaging indicator.
 \begin{lemma}\label{Section parameter LSM theorem}
 Assume that $k$ is not an interior transmission eigenvalue.
Suppose that $\{ \mathcal{R}_\alpha\}_{\alpha>0}$ is a family of regularization schemes given by \eqref{Section FM and a LSM R_alpha def}-- \eqref{Section FM and a LSM f_alpha prop} and set $g_{z,\alpha}  = \mathcal{R}_\alpha \phi_z$. Then it holds for any   $\overline{B(z,\epsilon)} \subset \Omega$ that
\begin{equation} \label{Section FM and a LSM main of LSM theorem LSM=FM}
\lim_{\alpha \to 0} \langle  g_{z,\alpha}, \phi_z  \rangle_{L^2(\mathbb{S}^{d-1})} = \langle (\mathcal{P}_\Omega \mathcal{T} \mathcal{P}_\Omega)^{-1} \mathcal{P}_\Omega E_z, \mathcal{P}_\Omega E_z  \rangle_{L^2(\Omega)}.
\end{equation}
\end{lemma}
\begin{proof}

It is noted that
\begin{equation} \label{Section FM and a LSM main of LSM theorem proof eqn1}
\langle  g_{z,\alpha}, \phi_z  \rangle_{L^2(\mathbb{S}^{d-1})} = \langle  g_{z,\alpha},\mathcal{H}^* E_z \rangle_{L^2(\mathbb{S}^{d-1})} = \langle  \mathcal{H} g_{z,\alpha}, E_z  \rangle_{L^2(\Omega)},
\end{equation}
then we first proceed with finding an expression for $\mathcal{H} g_{z,\alpha} \in Y_\Omega$. 

(a). Indeed from Lemma \ref{Section parameter LSM lemma PH*HP}, we can find the eigensysten expansion 
$$
\mathcal{P}_\Omega \mathcal{T}\mathcal{P}_\Omega\mathcal{H} g_{z,\alpha} = \sum_{j=0}^\infty \langle \mathcal{P}_\Omega \mathcal{T}\mathcal{P}_\Omega\mathcal{H} g_{z,\alpha} , \psi_j \rangle_{L^2(\Omega)} \psi_j,
$$
note that $ \mathcal{P}_\Omega \mathcal{T}\mathcal{P}_\Omega$ has a bounded inverse (cf. Proposition \ref{Section operator middle positive definite prop}), then
\begin{equation} \label{Section FM and a LSM main of LSM theorem proof eqn2}
\mathcal{H} g_{z,\alpha} = \sum_{j=0}^\infty \langle \mathcal{P}_\Omega \mathcal{T}\mathcal{P}_\Omega\mathcal{H} g_{z,\alpha} , \psi_j \rangle_{L^2(\Omega)} \big( \mathcal{P}_\Omega \mathcal{T}\mathcal{P}_\Omega \big)^{-1}\psi_j,
\end{equation}
this together with \eqref{Section FM and a LSM main of LSM theorem proof eqn1} and \eqref{Section FM and a LSM main of LSM theorem proof eqn2} allows us to obtain that
\begin{eqnarray} \label{Section FM and a LSM main of LSM theorem proof eqn3}
\langle  g_{z,\alpha}, \phi_z  \rangle_{L^2(\mathbb{S}^{d-1})}  &=&  \sum_{j=0}^\infty \langle \mathcal{P}_\Omega \mathcal{T}\mathcal{P}_\Omega\mathcal{H} g_{z,\alpha} , \psi_j \rangle_{L^2(\Omega)} \langle \big( \mathcal{P}_\Omega \mathcal{T}\mathcal{P}_\Omega \big)^{-1}\psi_j, E_z\rangle_{L^2(\Omega)} \nonumber \\
&=&  \sum_{j=0}^\infty \langle \mathcal{P}_\Omega \mathcal{T}\mathcal{P}_\Omega\mathcal{H} g_{z,\alpha} , \psi_j \rangle_{L^2(\Omega)} \langle \big( \mathcal{P}_\Omega \mathcal{T}\mathcal{P}_\Omega \big)^{-1}\psi_j, \mathcal{P}_\Omega E_z\rangle_{L^2(\Omega)} \nonumber \\
&=& \sum_{j=0}^\infty \langle \mathcal{P}_\Omega \mathcal{T}\mathcal{P}_\Omega\mathcal{H} g_{z,\alpha} , \psi_j \rangle_{L^2(\Omega)} \langle  \psi_j, \big[\big( \mathcal{P}_\Omega \mathcal{T}\mathcal{P}_\Omega \big)^{-1} \big]^*\big(\mathcal{P}_\Omega E_z\big)\rangle_{L^2(\Omega)}. \qquad
\end{eqnarray}
 From Lemma \ref{Section parameter LSM lemma} we have that $\| \mathcal{H} g_{z,\alpha}\|_{L^2(\Omega)}$ is uniformly bounded with respect to $\alpha$, note that $\big[\big( \mathcal{P}_\Omega \mathcal{T}\mathcal{P}_\Omega \big)^{-1} \big]^*\big(\mathcal{P}_\Omega E_z\big) \in L^2(\Omega)$ whereby $ \sum_{j=0}^\infty   \langle  \psi_j, \big[\big( \mathcal{P}_\Omega \mathcal{T}\mathcal{P}_\Omega \big)^{-1} \big]^*\big(\mathcal{P}_\Omega E_z\big)\rangle_{L^2(\Omega)}$ is convergent (and this series is independent of $\alpha$), we can obtain that the infinite series in  \eqref{Section FM and a LSM main of LSM theorem proof eqn3} is uniformly convergent. This allows to apply the dominated convergence theorem to obtain that
 {\small
 \begin{equation} \label{Section FM and a LSM main of LSM theorem proof eqn4}
\lim_{\alpha \to 0}\langle  g_{z,\alpha}, \phi_z  \rangle_{L^2(\mathbb{S}^{d-1})} 
=\sum_{j=0}^\infty \lim_{\alpha \to 0}\langle \mathcal{P}_\Omega \mathcal{T}\mathcal{P}_\Omega\mathcal{H} g_{z,\alpha} , \psi_j \rangle_{L^2(\Omega)} \langle  \psi_j, \big[\big( \mathcal{P}_\Omega \mathcal{T}\mathcal{P}_\Omega \big)^{-1} \big]^*\big(\mathcal{P}_\Omega E_z\big)\rangle_{L^2(\Omega)}.  
\end{equation}
}

(b). Now we proceed to compute the limit in \eqref{Section FM and a LSM main of LSM theorem proof eqn4}. We first derive 
\begin{eqnarray}\label{Section FM and a LSM main of LSM theorem proof eqn5}
&& \langle \mathcal{P}_\Omega \mathcal{T}\mathcal{P}_\Omega\mathcal{H} g_{z,\alpha} , \psi_j \rangle_{L^2(\Omega)} \overset{\eqref{Section parameter LSM lemma PH*HP eigensystem}}{=} \langle \mathcal{P}_\Omega \mathcal{T}\mathcal{P}_\Omega\mathcal{H} g_{z,\alpha} ,\frac{1}{|\lambda_j|^2}\mathcal{P}_\Omega\mathcal{H} \mathcal{H}^* \mathcal{P}_\Omega \psi_j  \rangle_{L^2(\Omega)} \nonumber \\
&=&  \langle \mathcal{H}^* \mathcal{P}_\Omega \mathcal{T}\mathcal{P}_\Omega\mathcal{H} g_{z,\alpha} ,\frac{1}{|\lambda_j|^2}   \mathcal{H}^* \mathcal{P}_\Omega \psi_j    \rangle_{L^2(\Omega)} = \langle \mathcal{F} g_{z,\alpha} ,\frac{1}{|\lambda_j|^2}   \mathcal{H}^* \mathcal{P}_\Omega \psi_j    \rangle_{L^2(\Omega)}. 
\end{eqnarray}
Again from the definition of $g_{z,\alpha}$,  one gets 
\begin{equation*}  
g_{z,\alpha} =  \sum_{n=0}^\infty f_\alpha(\mu_n^2) \overline{\mu_n} \langle \phi_z, \zeta_n \rangle \zeta_n \quad \mbox{and} \quad \mathcal{F} g_{z,\alpha} =  \sum_{n=0}^\infty f_\alpha(\mu_n^2)  |\mu_n|^2 \langle \phi_z, \zeta_n \rangle \zeta_n,
\end{equation*}
note that $\sum_{n=0}^\infty f_\alpha(\mu_n^2)  |\mu_n|^2 \langle \phi_z, \zeta_n \rangle \zeta_n$ is uniformly convergent respect to $\alpha$, then
\begin{equation*}  
 \lim_{\alpha \to 0}\mathcal{F} g_{z,\alpha} =  \sum_{n=0}^\infty  \lim_{\alpha \to 0} \big(f_\alpha(\mu_n^2)  |\mu_n|^2\big) \langle \phi_z, \zeta_n \rangle \zeta_n = \sum_{n=0}^\infty  \langle \phi_z, \zeta_n \rangle \zeta_n = \phi_z,
\end{equation*}
this limit allows to compute the limit of \eqref{Section FM and a LSM main of LSM theorem proof eqn5} as $\alpha \to 0$,
\begin{eqnarray} \label{Section FM and a LSM main of LSM theorem proof eqn6}
&&\lim_{\alpha \to 0}\langle \mathcal{P}_\Omega \mathcal{T}\mathcal{P}_\Omega\mathcal{H} g_{z,\alpha} , \psi_j \rangle_{L^2(\Omega)} =\langle \lim_{\alpha \to 0} \mathcal{F} g_{z,\alpha} ,\frac{1}{|\lambda_j|^2}   \mathcal{H}^* \mathcal{P}_\Omega \psi_j   \rangle_{L^2(\Omega)} \nonumber \\
&=& \langle \phi_z ,\frac{1}{|\lambda_j|^2}   \mathcal{H}^* \mathcal{P}_\Omega \psi_j   \rangle_{L^2(\Omega)} = \langle \mathcal{H}^* E_z ,\frac{1}{|\lambda_j|^2}   \mathcal{H}^* \mathcal{P}_\Omega \psi_j   \rangle_{L^2(\Omega)} \nonumber\\ 
&=& \langle   E_z ,\frac{1}{|\lambda_j|^2}   \mathcal{H} \mathcal{H}^* \mathcal{P}_\Omega \psi_j   \rangle_{L^2(\Omega)} = \langle   E_z ,  \psi_j   \rangle_{L^2(\Omega)} .
\end{eqnarray}
Then we can compute the limit in \eqref{Section FM and a LSM main of LSM theorem proof eqn4} by the above equation \eqref{Section FM and a LSM main of LSM theorem proof eqn6} to obtain that
 \begin{equation} \label{Section FM and a LSM main of LSM theorem proof eqn7}
\lim_{\alpha \to 0}\langle  g_{z,\alpha}, \phi_z  \rangle_{L^2(\mathbb{S}^{d-1})} 
=\sum_{j=0}^\infty \langle   E_z ,  \psi_j   \rangle_{L^2(\Omega)} \langle  \psi_j, \big[\big( \mathcal{P}_\Omega \mathcal{T}\mathcal{P}_\Omega \big)^{-1} \big]^*\big(\mathcal{P}_\Omega E_z\big)\rangle_{L^2(\Omega)}.  
\end{equation}
 
(c). We finally show the right hand side of \eqref{Section FM and a LSM main of LSM theorem proof eqn7} is equal to the right hand side of \eqref{Section FM and a LSM main of LSM theorem LSM=FM}. This is due to that
 \begin{eqnarray*} \label{Section FM and a LSM main of LSM theorem proof eqn8}
&& \langle [(\mathcal{P}_\Omega \mathcal{T} \mathcal{P}_\Omega)^{-1} ] \mathcal{P}_\Omega E_z, \mathcal{P}_\Omega E_z  \rangle_{L^2(\Omega)} =  \langle   \mathcal{P}_\Omega E_z, \big[\big( \mathcal{P}_\Omega \mathcal{T}\mathcal{P}_\Omega \big)^{-1} \big]^* \mathcal{P}_\Omega E_z  \rangle_{L^2(\Omega)}\\
&=& \sum_{j=0}^\infty \langle   \mathcal{P}_\Omega E_z ,  \psi_j   \rangle_{L^2(\Omega)} \langle  \psi_j, \big[\big( \mathcal{P}_\Omega \mathcal{T}\mathcal{P}_\Omega \big)^{-1} \big]^*\big(\mathcal{P}_\Omega E_z\big)\rangle_{L^2(\Omega)} \\
&=& \sum_{j=0}^\infty \langle    E_z ,  \psi_j   \rangle_{L^2(\Omega)} \langle  \psi_j, \big[\big( \mathcal{P}_\Omega \mathcal{T}\mathcal{P}_\Omega \big)^{-1} \big]^*\big(\mathcal{P}_\Omega E_z\big)\rangle_{L^2(\Omega)}.  
\end{eqnarray*}
Combining the above equation and \eqref{Section FM and a LSM main of LSM theorem proof eqn7} proves \eqref{Section FM and a LSM main of LSM theorem LSM=FM}. This completes the proof.
\end{proof}
\begin{remark}
\label{remarkinjectivityP}
The projection   $\mathcal{P}_\Omega E_z$ is unique, even though $E_z$ is not unique as is seen from the constructive proof of Lemma \ref{Section FM and a LSM lemma phiz S Omega}. This is due to that for any $E_z, \widetilde{E}_z\in L^2(\Omega)$ such that 
$$
\phi_z = \mathcal{H}^* E_z =\mathcal{H}^* \widetilde{E}_z,
$$
we can obtain
$$
E_z - \widetilde{E}_z = \mbox{Null}(\mathcal{H}^*) = Y_\Omega^\perp \quad\mbox{ thereby }\quad \mathcal{P}_\Omega E_z =\mathcal{P}_\Omega \widetilde{E}_z.
$$
The next lemma gives a more explicit expression of $\mathcal{P}_\Omega E_z$.
\end{remark}
 
 \begin{theorem} \label{Section FM and a LSM main of LSM theorem lemma PomegaEz}
 Let $E_z^{\mathcal{P}}:=\mathcal{P}_\Omega E_z$. Then $E_z^{\mathcal{P}} = -(\Delta + k^2) w_z^{\mathcal{P}}$ in $\Omega$, where $w_z^{\mathcal{P}} \in H^2(\Omega)$ has to satisfy the fourth-order boundary value problem
 \begin{eqnarray}
(\Delta + k^2)(\Delta + k^2) w_z^{\mathcal{P}} = 0 &\mbox{ in }& \Omega, \label{Section FM and a LSM main of LSM theorem lemma PomegaEz fourthorder eqn1}\\
w_z^{\mathcal{P}} =\Phi(\cdot,z)&\mbox{ on }& \partial \Omega,\label{Section FM and a LSM main of LSM theorem lemma PomegaEz fourthorder eqn2}\\
\frac{\partial  w_z^{\mathcal{P}} }{\partial \nu}=\frac{\partial\Phi(\cdot,z)}{\partial \nu}&\mbox{ on }& \partial \Omega.\label{Section FM and a LSM main of LSM theorem lemma PomegaEz fourthorder eqn3}
\end{eqnarray}
 \end{theorem}
 \begin{proof}
Recall in Theorem \ref{Section FM and a LSM lemma phiz S Omega} that $w_z(x)= \chi_{|x-z|}\Phi(x,z)$ in $\mathbb{R}^d$ where $\chi \in C^\infty(\mathbb{R})$ is a cut off function with $\chi(t)=1$ for $|t| \ge \epsilon$ and $\chi(t)=0$ for $|t| \le \epsilon/2$, and $E_z = -(\Delta +k^2) w_z$.
 Set $E_z^{\mathcal{P}}:=\mathcal{P}_\Omega E_z$, $w_z^{\mathcal{P}}:=\int_{\Omega}\Phi(\cdot,y)E_z^{\mathcal{P}}(y) \ind y$ and it follows that $(\Delta +k^2)w_z^{\mathcal{P}} = -E_z^{\mathcal{P}}$. By elliptic regularity \cite[Theorem 4.16]{mclean00}, $w_z^{\mathcal{P}}$ and $w_z$ are functions in $H^2_{loc}(\mathbb{R}^d)$.
Since $E_z^{\mathcal{P}}=\mathcal{P}_\Omega E_z$, then  $\mathcal{H} E_z^{\mathcal{P}}=\mathcal{H}\mathcal{P}_\Omega E_z=\mathcal{H} E_z$, i.e., the far-field patterns of $w_z^{\mathcal{P}}$ and of $w_z$ coincide, thereby by unique continuation $w_z^{\mathcal{P}}=w_z$ outside $\Omega$ and hence
 \begin{eqnarray*}
w_z^{\mathcal{P}} =w_z=\chi_{|\cdot-z|}\Phi(\cdot,z)=\Phi(\cdot,z)&\mbox{ on }& \partial \Omega,\\
\frac{\partial  w_z^{\mathcal{P}} }{\partial \nu}=\frac{\partial  w_z }{\partial \nu} = \frac{\partial ( \chi_{|\cdot-z|}\Phi(\cdot,z))}{\partial \nu}=\frac{\partial\Phi(\cdot,z)}{\partial \nu}&\mbox{ on }& \partial \Omega.
\end{eqnarray*}
 This shows \eqref{Section FM and a LSM main of LSM theorem lemma PomegaEz fourthorder eqn2}--\eqref{Section FM and a LSM main of LSM theorem lemma PomegaEz fourthorder eqn3}.
Note further that $E_z^{\mathcal{P}}=\mathcal{P}_\Omega E_z \in  Y_\Omega$, then  $E_z^{\mathcal{P}}$ satisfies the Helmholtz equation in the distributional sense and hence $(\Delta + k^2)(\Delta + k^2) w_z^{\mathcal{P}} =- (\Delta + k^2) E_z^{\mathcal{P}}= 0$ in the distributional sense, this shows \eqref{Section FM and a LSM main of LSM theorem lemma PomegaEz fourthorder eqn1}. 

{ We now complete the lemma  by showing that the solution is unique. Otherwise suppose that $u \in H^2(\Omega)$ satisfying
\begin{eqnarray*}
(\Delta + k^2)(\Delta + k^2) u = 0 &\mbox{ in }& \Omega,
\end{eqnarray*}
with $u|_{\partial \Omega}=0$ and $\frac{\partial u}{\partial \nu}|_{\partial \Omega}=0$, then integrating by parts twice yields that
\begin{eqnarray*}
    \int_\Omega [(\Delta + k^2)(\Delta + k^2) u] \overline{u} {\, \rm d} x = \int_\Omega [(\Delta + k^2) u] \overline{(\Delta + k^2)u} {\, \rm d} x
\end{eqnarray*}
and therefore $(\Delta + k^2)u=0$. Note that $u|_{\partial \Omega}=0$ and $\frac{\partial u}{\partial \nu}|_{\partial \Omega}=0$, then $u$ has to vanish by the Green's formula. This completes the proof.}
 \end{proof}
 Now we are ready to prove the main theorem.
  \begin{theorem}\label{Section parameter LSM theorem EzP}
 Assume that $k$ is not an interior transmission eigenvalue.
Suppose that $\{ \mathcal{R}_\alpha\}_{\alpha>0}$ is a family of regularization schemes given by \eqref{Section FM and a LSM R_alpha def}-- \eqref{Section FM and a LSM f_alpha prop} and set $g_{z,\alpha}  = \mathcal{R}_\alpha \phi_z$. Then it holds for any $z\in \Omega$  that
\begin{equation} \label{Section FM and a LSM main of LSM theorem LSM=FM EzP}
\lim_{\alpha \to 0} \langle  g_{z,\alpha}, \phi_z  \rangle_{L^2(\mathbb{S}^{d-1})} = \langle (\mathcal{P}_\Omega \mathcal{T} \mathcal{P}_\Omega)^{-1}   E_z^{\mathcal{P}}, E_z^{\mathcal{P}}  \rangle_{L^2(\Omega)},
\end{equation}
where $E_z^{\mathcal{P}} = -(\Delta + k^2) w_z^{\mathcal{P}}$ in $\Omega$ and $w_z^{\mathcal{P}} \in H^2(\Omega)$ is required to satisfy  \eqref{Section FM and a LSM main of LSM theorem lemma PomegaEz fourthorder eqn1}--\eqref{Section FM and a LSM main of LSM theorem lemma PomegaEz fourthorder eqn3}.
\end{theorem}
\begin{proof}
For any $z \in \Omega$, let  $\epsilon$ (which depends on $z$)  be such that  $\overline{B(z,\epsilon)} \subset \Omega$, then  Lemma \ref{Section parameter LSM theorem} yields that
\eqref{Section FM and a LSM main of LSM theorem LSM=FM}. Note from Lemma \ref{Section FM and a LSM main of LSM theorem lemma PomegaEz} that $E_z^{\mathcal{P}}=\mathcal{P}_\Omega E_z$ where $E_z^{\mathcal{P}} = -(\Delta + k^2) w_z^{\mathcal{P}}$ in $\Omega$ and $w_z^{\mathcal{P}} \in H^2(\Omega)$ is required to satisfy \eqref{Section FM and a LSM main of LSM theorem lemma PomegaEz fourthorder eqn1}--\eqref{Section FM and a LSM main of LSM theorem lemma PomegaEz fourthorder eqn3}. This completes the proof.
\end{proof} 
Under certain regularity assumption, we prove the  well-posedness of problem \eqref{Section FM and a LSM main of LSM theorem lemma PomegaEz fourthorder eqn1}-\eqref{Section FM and a LSM main of LSM theorem lemma PomegaEz fourthorder eqn3} as follows.
\begin{lemma} \label{lemma existence 4th order problem}
Assume that $\partial \Omega$ is $C^{1,1}$.
For $(f,g) \in H^{3/2}(\partial \Omega) \times H^{1/2}(\partial \Omega)$, the following fourth-order boundary value problem has a unique solution $w\in H^2(\Omega)$ where
 \begin{eqnarray}
(\Delta + k^2)(\Delta + k^2) w = 0 &\mbox{ in }& \Omega, \label{fourthorderprobelmwellposedeqn1}\\
w =f&\mbox{ on }& \partial \Omega, \label{fourthorderprobelmwellposedeqn2}\\
\frac{\partial  w }{\partial \nu}=g&\mbox{ on }& \partial \Omega.\label{fourthorderprobelmwellposedeqn3}
\end{eqnarray}
and we have that
$$
\left\|w\right\|_{H^2(\Omega)} \leq c\big(\left\|f\right\|_{H^{\frac{3}{2}}(\Omega)}+\left\|g\right\|_{H^{\frac{1}{2}}(\Omega)}\big)
$$
for some positive constant $c$.
\end{lemma}
\begin{proof}
First let $\theta\in H^2(\Omega)$ be a lifting function \cite{mclean00} such that $\theta=f$ and $\partial \theta/\partial\nu =g$ on $\partial \Omega$  and $\left\|\theta\right\|_{H^2(\Omega)}\leq c(\left\|f\right\|_{H^{\frac{3}{2}}(\Omega)}+\left\|g\right\|_{H^{\frac{1}{2}(}\Omega)})$ for some constant $c>0$.
Let $w_0:=w-\theta$, then $w_0 \in H^2_0(\Omega)$ and we can write \eqref{fourthorderprobelmwellposedeqn1}-\eqref{fourthorderprobelmwellposedeqn3} as an equivalent variational formulation to find $w_0\in H^2_0(\Omega)$ such that 
\begin{equation*}
\int_\Omega (\Delta + k^2) w_0(\Delta + k^2) \bar \psi \ind x = \int_\Omega (\Delta + k^2) \theta(\Delta + k^2) \bar \psi \ind x
\end{equation*}
for all $\psi\in H^2_0(\Omega)$. By Riesz representation theorem, there exist a unique $l\in H^2_0(\Omega)$ such that the right hand side is equal to $<l,\psi>_{H^2_0(\Omega)}$ and $\left\| l \right\|_{H^2(\Omega)}\leq c\left\|\theta\right\|_{H^2(\Omega)}$. The sesquilinear form on the left is coercive because there exist some positive constant $c$ such that  $c\left\| u  \right\|_{H^2(\Omega)}\leq\left\|\Delta   u \right\|_{L^2(\Omega)}$ for all $u\in H^2_0(\Omega)$ \cite{Grisvardbook}, the uniqueness of the problem \eqref{fourthorderprobelmwellposedeqn1}-\eqref{fourthorderprobelmwellposedeqn3} and Fredholm theory. This allows to prove that 
$
\left\|w_0\right\|_{H^2(\Omega)} \leq  c \|\theta\|_{H^2(\Omega)}
$ for some positive constant $c$
and consequently $\left\|w\right\|_{H^2(\Omega)} \leq  c \|\theta\|_{H^2(\Omega)}$ for some (other) positive constant $c$. This completes the proof.
\end{proof}

Under the assumption that $\partial \Omega$ is $C^{1,1}$, then it follows Lemma \ref{lemma existence 4th order problem} that there exists a unique solution $w_z^{\mathcal{P}} \in H^2(\Omega)$  to \eqref{Section FM and a LSM main of LSM theorem lemma PomegaEz fourthorder eqn1}--\eqref{Section FM and a LSM main of LSM theorem lemma PomegaEz fourthorder eqn3}. This leads to the following Corollary.
\begin{corollary}
 Assume that $k$ is not an interior transmission eigenvalue and that $\partial \Omega$ is $C^{1,1}$.
Suppose that $\{ \mathcal{R}_\alpha\}_{\alpha>0}$ is a family of regularization schemes given by \eqref{Section FM and a LSM R_alpha def}-- \eqref{Section FM and a LSM f_alpha prop} and set $g_{z,\alpha}  = \mathcal{R}_\alpha \phi_z$. Then it holds for any $z\in \Omega$  that
\begin{equation} \label{Section FM and a LSM main of LSM theorem LSM=FM EzP}
\lim_{\alpha \to 0} \langle  g_{z,\alpha}, \phi_z  \rangle_{L^2(\mathbb{S}^{d-1})} = \langle (\mathcal{P}_\Omega \mathcal{T} \mathcal{P}_\Omega)^{-1}   E_z^{\mathcal{P}}, E_z^{\mathcal{P}}  \rangle_{L^2(\Omega)},
\end{equation}
where $E_z^{\mathcal{P}} = -(\Delta + k^2) w_z^{\mathcal{P}}$ in $\Omega$ and $w_z^{\mathcal{P}} \in H^2(\Omega)$ is the unique solution to  \eqref{Section FM and a LSM main of LSM theorem lemma PomegaEz fourthorder eqn1}--\eqref{Section FM and a LSM main of LSM theorem lemma PomegaEz fourthorder eqn3}.    
\end{corollary}
\begin{remark}
It is possible to impose less regularity assumptions on $(f,g)$. For example, according to \cite{LionsEnrico1961}, for $(f,g) \in H^{-1/2}(\partial \Omega) \times H^{-3/2}(\partial \Omega)$, the following problem has a unique very weak solution $w$ in $L^2(\Omega)$ where
 \begin{eqnarray*}
(\Delta + k^2)(\Delta + k^2) w = 0 &\mbox{ in }& \Omega, \\
w =f&\mbox{ on }& \partial \Omega,\\
\frac{\partial  w }{\partial \nu}=g&\mbox{ on }& \partial \Omega.
\end{eqnarray*}
\end{remark}
\subsection{The Born case}

In the Born approximation regime, the Born approximation $u^s_b \in H^1_{loc}(\mathbb{R}^d)$ is the unique radiating solution  to
\begin{equation} \label{Born medium us eqn1+2}
\Delta u^s_b + k^2  u^s_b = -k^2 q f,
\end{equation}
where $f(x)=e^{ikx\cdot \hat{\theta}}$. The forward scattering problem \eqref{Born medium us eqn1+2} obviously has a unique radiating solution. Similarly this gives the Born far-field pattern $u^{\infty}_b(\hat{x};\hat{\theta};k)$ according to
\begin{equation*} 
u^{s}_b(x;\hat{\theta};k)
=
\left\{
\begin{array}{cc}
\frac{e^{i\frac{\pi}{4}}}{\sqrt{8k\pi}} \frac{e^{ikr}}{\sqrt{r}}\left\{u^{\infty}_b(\hat{x};\hat{\theta};k)+\mathcal{O}\left(\frac{1}{r}\right)\right\}  & d=2        \\
 \frac{e^{ik |x|}}{4\pi |x|} \left\{ u^\infty_b (\hat{x};\hat{\theta};k) +     \mathcal{O}\left(\frac{1}{r}\right) \right\}& d=3
\end{array}
\right.
 \quad\mbox{as }\,r=|x|\rightarrow\infty,
\end{equation*}
uniformly with respect to all directions $\hat{x}:=x/|x|\in\mathbb{S}^{d-1}$.

If we introduce a Born far-field operator $\mathcal{F}_b: L^2(\mathbb{S}^{d-1}) \to L^2(\mathbb{S}^{d-1})$  by 
\begin{equation} \label{Section operator N def Born}
\left( \mathcal{F}_b g \right)(\hat{x}) := \int_{\mathbb{S}^{d-1}} u^{\infty}_b(\hat{x};\hat{\theta};k) g(\hat{\theta}) \ind s(\hat{\theta}),
\end{equation}
then we immediately arrive at the factorization $\mathcal{F}_b = \mathcal{H}^* \mathcal{T}_{b} \mathcal{H}$  
where $\mathcal{T}_{b}: L^2(\Omega) \to L^2(\Omega)$ is given by \eqref{Section operator T_omega def Born}.

In the Born  model and our hypothesis on $q$, the middle operator $\mathcal{T}_{b}$ is linear, self-adjoint, and positive definite and $\mathcal{T}_{b} $ is coercive as an operator from $Y_\Omega \to Y_\Omega$. Thereby  Theorem \ref{Section parameter LSM theorem EzP} applies immediately and we have the following Corollary. 
 
  \begin{corollary}\label{Section parameter LSM corollary Born}
Suppose that $\{ \mathcal{R}_{b,\alpha}\}_{\alpha>0}$ is a family of regularization schemes given by \eqref{Section FM and a LSM R_alpha def}-- \eqref{Section FM and a LSM f_alpha prop} where we replace the eigensystem of $\mathcal{F}$ by the eigensystem of $\mathcal{F}_b$ \eqref{Section operator N def Born},
 and set $g_{b,z,\alpha}  = \mathcal{R}_{b,\alpha} \phi_z$. Then it holds   for any $z\in \Omega$   that
\begin{equation} \label{Section FM and a LSM main of LSM theorem LSM=FM Born}
\lim_{\alpha \to 0} \langle  g_{b,z,\alpha}, \phi_z  \rangle_{L^2(\mathbb{S}^{d-1})}  = \langle (\mathcal{P}_\Omega \mathcal{T}_b \mathcal{P}_\Omega)^{-1}   E_z^{\mathcal{P}}, E_z^{\mathcal{P}}  \rangle_{L^2(\Omega)}.
\end{equation}
where $E_z^{\mathcal{P}} = -(\Delta + k^2) w_z^{\mathcal{P}}$ in $\Omega$ with $w_z^{\mathcal{P}} \in H^2(\Omega)$ being the unique solution to \eqref{Section FM and a LSM main of LSM theorem lemma PomegaEz fourthorder eqn1}--\eqref{Section FM and a LSM main of LSM theorem lemma PomegaEz fourthorder eqn3}.

In particular, if $q>0$ is a constant, then  
\begin{equation*} 
\lim_{\alpha \to 0} \langle  g_{b,z,\alpha}, \phi_z  \rangle_{L^2(\mathbb{S}^{d-1})}  = \frac{1}{q} \| E_z^{\mathcal{P}}\|^2_{L^2(\Omega)}.
\end{equation*}
\end{corollary}

\begin{remark}
Corollary \ref{Section parameter LSM corollary Born} allows to reconstruct $\langle (\mathcal{P}_\Omega \mathcal{T}_b \mathcal{P}_\Omega)^{-1}   E_z^{\mathcal{P}}, E_z^{\mathcal{P}}  \rangle_{L^2(\Omega)}$  which becomes, in the case of constant contrast, a weighed average of $1/q$ given by $\frac{1}{q} \| E_z^{\mathcal{P}}\|^2_{L^2(\Omega)}$. 
It is noted in \cite{audibertmeng23} that the Born model can  reconstruct the average of $1/q$ over a small region $B(z,\epsilon)$ since the data operator (and the corresponding $\phi_z$) in \cite{audibertmeng23} was reformulated in another different way.    
\end{remark}

\section{Interpretation via the generalized linear sampling method and alternative characterizations} \label{Section GLSM interpretation}

\subsection{Interpretation via the generalized linear sampling method}
There is a deep connection between Theorem \ref{Section parameter LSM theorem EzP} and the generalized linear sampling method \cite{audiberthaddar15}. We first recall the main ingredients of the generalized linear sampling method. Consider the following cost functional 
$$J_\alpha(g, \phi_z ) :=  \alpha \left\| (\mathcal{F^*F})^{1/4}g \right\|^2+ \left\| \mathcal{F}g - \phi_z\right\|^2 $$
and a minimizing sequence $g^*_{z,\alpha}$ such that 
$$
J_\alpha(g^*_{z,\alpha}, \phi_z ) \leq \inf_g J_\alpha(g, \phi_z ) + p(\alpha)
$$
with $p(\alpha)/\alpha\to 0$ when $\alpha\to 0$. {For this general $p(\alpha)$, one can show the existence of the minimizing sequence $g^*_{z,\alpha}$  \cite{audiberthaddar15}.}

Assume that $k$ is not an interior transmission eigenvalue, then there exists a unique solution $(w_z,v_z) \in L^2(\Omega) \times L^2(\Omega)$  \cite{cakoni2016inverse} such that $w_z-v_z \in H^2(\Omega)$  and
\begin{eqnarray}
\Delta w_z + k^2(1+q) w_z =0 &\mbox{ in }& \Omega, \label{SolLSM-ITEP-1}\\ 
\Delta v_z + k^2v_z =0 &\mbox{ in }& \Omega,\label{SolLSM-ITEP-2}\\
w_z - v_z = \Phi(\cdot,z) &\mbox{ on }& \partial \Omega, \label{SolLSM-ITEP-3}\\
\frac{\partial w_z}{\partial \nu} - \frac{\partial v_z  }{\partial \nu} = \frac{\partial \Phi(\cdot,z) }{\partial \nu} &\mbox{ on }& \partial \Omega. \label{SolLSM-ITEP-4}
\end{eqnarray}
To discuss several properties of the generalized linear sampling method, recall the operators $\mathcal{H}^*$ and $\mathcal{T}$ given by \eqref{Section operator S*_omega def} and \eqref{Section operator T_omega def}, respectively,  we introduce the operator $\mathcal{G} = \mathcal{H}^* \mathcal{T}$, then it follows that $\mathcal{G} $ is compact, $\mathcal{F}$ has dense range, and $ \left|  \langle  (\mathcal{F^*F})^{1/2}g,g  \rangle  \right| \geq  \mu \left\|  \mathcal{H} g \right\|^2$ for some positive constant $\mu$, e.g. \cite{audiberthaddar15}. 
Moreover, we can see from \eqref{SolLSM-ITEP-1}--\eqref{SolLSM-ITEP-4} that the incident field $v_z$ can be seen as an incident field which induces a scattered field $\Phi(\cdot, z)$ outside $\Omega$   and a total field $w_z$ inside $\Omega$; this allows to conclude that $\mathcal{G} v_z = \phi_z$ since the operator $\mathcal{G}$ maps an incident field to the corresponding far-field. From \eqref{Section operator T_omega def} we can directly deduce that  $\mathcal{T}v_z=k^2qw_z$, which implies $\mathcal{H}^*(k^2qw_z) = \phi_z$. Note that $ \mathcal{P}_\Omega$ is the projection onto $\overline{R(\mathcal{H})}$ and  $\mathcal{H}^*E^{\mathcal{P}}_z = \phi_z$ we deduce as explained in remark \ref{remarkinjectivityP} that $ E_z^{\mathcal{P}} = \mathcal{P}_\Omega (k^2qw_z)$.

Now we are ready to recall the following result  from \cite{audiberthaddar15}. We also refer to \cite[Section 7.2]{kirsch2008factorization}.
\begin{theorem}\label{Section generalized linear sampling method theorem}
Let $\mathcal{F}$,  $\mathcal{H}$, and $\mathcal{T}$ be given by \eqref{Section operator N def}, \eqref{Section operator S_omega def}, and \eqref{Section operator T_omega def}, respectively.  
Then the following properties hold.
\begin{itemize}
\item  If $z \in \mathbb{R}^d \backslash \overline{\Omega}$, then $\limsup_{\alpha\to 0}\left|  \langle  (\mathcal{F^*F})^{1/2}g^*_{z,\alpha},g^*_{z,\alpha}  \rangle  \right| = \infty$.
\item  If $z \in  \Omega$, then $\limsup_{\alpha\to 0}\left|  \langle  (\mathcal{F^*F})^{1/2}g^*_{z,\alpha},g^*_{z,\alpha}  \rangle  \right| \leq \infty$.
\item  Moreover $\mathcal{H} g^*_{z,\alpha}$ strongly converge in $L^2(\Omega)$ to $v_z$ in $Y_\Omega$.
\end{itemize}
As a direct consequence when $z \in \Omega$, it follows that 
 \begin{equation*}
\lim_{\alpha \to 0} \langle  g^*_{z,\alpha}, \phi_z  \rangle_{L^2(\mathbb{S}^{d-1})} = \langle v_z,\mathcal{P}_\Omega \mathcal{T} \mathcal{P}_\Omega v_z \rangle_{L^2(\Omega)}= \langle v_z,k^2qw_z \rangle_{L^2(\Omega)}.
\end{equation*}
\end{theorem}
\begin{proof}
Following \cite{audiberthaddar15}, $v_z \in Y_\Omega = \overline{R(\mathcal{H})}$ and that $\mathcal{P}_\Omega$ is the projection onto $Y_\Omega$, we have that
 \begin{equation*}
\lim_{\alpha \to 0} \langle  \mathcal{F} g^*_{z,\alpha}, g^*_{z,\alpha}  \rangle_{L^2(\mathbb{S}^{d-1})} = \langle v_z,\mathcal{P}_\Omega \mathcal{T} \mathcal{P}_\Omega v_z \rangle_{L^2(\Omega)}= \langle v_z,k^2qw_z \rangle_{L^2(\Omega)}.
\end{equation*}
Note the strong convergence of $\mathcal{F} g^*_{z,\alpha}$ to $\phi_z$, this proves the theorem.
\end{proof}

Now with the help of the  generalized linear sampling method, we provide a shorter, less constructive proof of Theorem \ref{Section parameter LSM theorem EzP} as follows.

\begin{proof}[Alternate Proof of Theorem \ref{Section parameter LSM theorem EzP}  ]
From lemma \ref{Section parameter LSM lemma} we know that $\mathcal{H}  g_{z,\alpha}$ is bounded as $\alpha$ goes to $0$. Therefore we can conclude that there is a weakly convergent subsequence. Moreover we know that $\mathcal{F}  g_{z,\alpha}$ strongly converges because $\mathcal{F}$ is a compact operator. Since $\mathcal{F}$ has dense range and due to classical regularization theory results, if it converges $\mathcal{F}g_{z,\alpha}$ has to converge to $\phi_z $. Therefore because of the injectivity of $ \mathcal{G}$ and the fact that $ \mathcal{G}v_z =\phi_z $, there is only one weak limit $v_z$ for the sequence $\mathcal{H}  g_{z,\alpha}$ as $\alpha$ goes to zero. Now one can derive that, as $\alpha\to 0$,
$$
\langle   g_{z,\alpha}, \phi_z  \rangle_{L^2(\mathbb{S}^{d-1})} =\langle  g_{z,\alpha},\mathcal{H}^*(k^2 q w_z) \rangle_{L^2(\mathbb{S}^{d-1})} = \langle  \mathcal{H} g_{z,\alpha}, k^2 q w_z  \rangle_{L^2(\Omega)}\to \langle v_z,k^2qw_z \rangle_{L^2(\Omega)}.
$$

The final part is to relate the right hand side of the above equality to $ E_z^{\mathcal{P}}$. To do that we recall first that $E_z^{\mathcal{P}} =\mathcal{P}_\Omega(k^2qw_z)$ and $\mathcal{T}v_z=k^2qw_z$, this allows to show that $\phi_z=\mathcal{H}^*E^\mathcal{P}_z = \mathcal{H}^*\mathcal{P}_\Omega\mathcal{T}\mathcal{P}_\Omega v_z  $. From the injectivity of $\mathcal{H}^*\mathcal{P}_\Omega$, we have that $E^\mathcal{P}_z =\mathcal{P}_\Omega\mathcal{T}\mathcal{P}_\Omega v_z$.
This proves that $\langle v_z,k^2qw_z \rangle_{L^2(\Omega)}$ is equal to $ \langle (\mathcal{P}_\Omega \mathcal{T} \mathcal{P}_\Omega)^{-1}   E_z^{\mathcal{P}}, E_z^{\mathcal{P}}  \rangle_{L^2(\Omega)}$ and completes the proof.
\end{proof}

\subsection{Alternative characterization of the indicator}
There are alternative ways to characterize the indicator by assuming that the background medium is $1+q$. More precisely, we first introduce the Green function for the medium  $1+q$ denoted  $\Phi_q (x,y)$ which is the radiating solution to 
$$
\Delta_x \Phi_q (\cdot,y) +k^2(1+q)\Phi_q (\cdot,y)  = \delta_y.
$$
Let $\Phi_q^\infty(\hat{x},y)$ be the far-field of $\Phi_q (\cdot,y)$.
In this way, the scattered field $u^s$ in \eqref{medium us eqn1+2} can also be given by
$$
u^\infty(\hat x)=\mathcal{H}^*_q(k^2qf) (\hat x) =  k^2 \int_\Omega \Phi_q^\infty(\hat x,y)q(y)   f(y) \ind y,
$$
where $\mathcal{H}^*_q: L^2(\Omega) \to  L^2(\mathbb{S}^{d-1})$ is defined by
$$
(\mathcal{H}^*_qh)(x) = \int_\Omega \Phi_q^\infty(\hat x,y)h(y) \ind y, \forall \quad h \in L^2(\Omega).
$$
Similarly to $\mathcal{H}^*$ given by \eqref{Section operator S*_omega def},  $\mathcal{H}^*_q$ is also not injective;   we can similarly introduce the projector $P_{q,\Omega}$ which projects a given function to the function space $Y_{q,\Omega}:=\{v\in L^2(\Omega): \ \Delta v+k^2(1+q)v=0\}$. The operator $\mathcal{H}^*_q P_{q,\Omega}$ is then injective and we can introduce the unique $E^\mathcal{P}_{q,z}$ given by $\mathcal{H}^*_q P_{q,\Omega}E^\mathcal{P}_{q,z} = \phi_z$. Again we have that  $E^\mathcal{P}_{q,z} =\Delta f_{q,z} +k^2(1+q)f_{q,z} $ where the function $f_{q,z}  \in L^2_\Delta(\Omega)$ is the unique solution to
$$
\left\{
\begin{array}{rcl}
(\Delta + k^2(1+q))(\Delta + k^2(1+q)) f_{q,z}  = 0 &\mbox{ in }& \Omega, \\
f_{q,z} =\Phi(\cdot,z)&\mbox{ on }& \partial \Omega,\\
\frac{\partial  f_{q,z} }{\partial \nu}=\frac{\partial\Phi(\cdot,z)}{\partial \nu}&\mbox{ on }& \partial \Omega.
\end{array}
\right.
$$
In the next theorem we give the relation between $E^\mathcal{P}_z$ and $E^\mathcal{P}_{q,z}$  and  obtain several equivalent representations of $\lim_{\alpha \to 0}\langle  g_{z,\alpha}, \phi_z  \rangle_{L^2(\mathbb{S}^{d-1})} 
$. 
\begin{theorem} \label{Section alternative characterization theorem}
If $k$ is not an interior transmission eigenvalue, let $(w_z,v_z)$ be given by the interior transmission problem \eqref{SolLSM-ITEP-1}-\eqref{SolLSM-ITEP-4}. Let $\mathcal{T}_{b}: L^2(\Omega) \to L^2(\Omega)$ be given by \eqref{Section operator T_omega def Born}.
It holds that 
\begin{equation}\label{eqn: E^P_z and E^P_qz}
 E^\mathcal{P}_{z}=  (P_{\Omega}\mathcal{T}P_{\Omega}) ( P_{q,\Omega} \mathcal{T}_b P_\Omega)^{-1}  E^\mathcal{P}_{q,z}.   
\end{equation}
As a consequence when $z \in \Omega$, we obtain the following equivalent representations of the imaging indicator $\lim_{\alpha \to 0}\langle  g_{z,\alpha}, \phi_z  \rangle_{L^2(\mathbb{S}^{d-1})} 
$ given by
\begin{itemize}
\item $ \langle  \mathcal{T} v_z,v_z\rangle = \langle E^\mathcal{P}_z,  ( P_{q,\Omega} \mathcal{T}_b P_\Omega)^{-1} E^\mathcal{P}_{q,z}\rangle$,
\item and  $ \langle  \mathcal{T} v_z,v_z\rangle =\langle k^2qw_z,v_z\rangle= \langle (P_{q,\Omega} \mathcal{T}_b P_{\Omega})^{-1}(P_{\Omega}\mathcal{T}P_{\Omega})(P_{q,\Omega}\mathcal{T}_bP_{\Omega})^{-1}E^\mathcal{P}_{q,z},E^\mathcal{P}_{q,z}\rangle$.
\end{itemize}
\end{theorem}

\begin{proof}
First we should prove that $( P_{q,\Omega} \mathcal{T}_b P_\Omega)^{-1} $ is well defined for $( P_{q,\Omega} \mathcal{T}_b P_\Omega) $ as an operator from $Y_\Omega$ to $Y_{q,\Omega}$. 
First for any $f\in Y_\Omega$ such that $( P_{q,\Omega} \mathcal{T}_b P_\Omega) f =0$, we will show that $f$ is trivial:  $\mathcal{H}^*_q P_{q,\Omega} \mathcal{T}_b P_\Omega f =0$ yields that $\mathcal{H}^*_q \mathcal{T}_b f =0$. Let $ g(x) = \int_\Omega \Phi_q(x,y)(k^2qb)(y) \ind y$, then it follows from $\mathcal{H}^*_q \mathcal{T}_b f = 0$ that the far-field $g^\infty=0$ and then $g(x) \equiv 0$ in $\mathbb{R}^d \backslash \overline{\Omega}$, this implies that $(f+g,f)$ satisfies the interior transmission problem in Definition \ref{Section operator ITE def}, then $f$ vanishes since $k$ is not an interior transmission eigenvalue. This shows that $P_{q,\Omega} \mathcal{T}_b P_\Omega$ is injective.
Second for any $h \in Y_{q,\Omega}$, we will show that there exist some $v \in Y_\Omega$ such that $P_{q,\Omega} \mathcal{T}_b P_\Omega v=h$: let $u(x) = \int_\Omega \Phi_q(x,y)h(y) \ind y$ for $x \in \mathbb{R}^d$; since $k$ is not an interior transmission eigenvalue, then there exist a unique solution $(w,v)\in Y_{q,\Omega}\times Y_{\Omega}$ that solves the interior transmission problem
\begin{eqnarray*}
\Delta w + k^2(1+q) w =0 &\mbox{ in }& \Omega,\\
\Delta v + k^2v =0 &\mbox{ in }& \Omega,\\
w - v = u &\mbox{ on }& \Omega,\\
\frac{\partial w}{\partial \nu} - \frac{\partial v}{\partial \nu} = \frac{\partial u}{\partial \nu} &\mbox{ on }& \partial \Omega,
\end{eqnarray*}
and one can understand above interior transmission problem by that the incident field $v$ induces scattering wave field $u$ and total wave field $w$, and as a result one can deduce that $u(x) = \int_\Omega \Phi_q(x,y)k^2q(y)v(y) \ind y$ for $x \in \mathbb{R}^d$.
This implies that $h = k^2 q v$, i.e., $P_{q,\Omega} \mathcal{T}_b P_\Omega v =h$. This demonstrates that $( P_{q,\Omega} \mathcal{T}_b P_\Omega)^{-1} $ is well defined.

To prove \eqref{eqn: E^P_z and E^P_qz}, recall that  $(w_z,v_z)$ satisfies the interior transmission problem \eqref{SolLSM-ITEP-1}-\eqref{SolLSM-ITEP-4}, then $v_z$ can be seen as an incident wave field which induces the far-field $\phi_z$, i.e., $\phi_z = \mathcal{G} v_z = \mathcal{H}^* \mathcal{P}_\Omega \mathcal{T} \mathcal{P}_\Omega v_z$, note that $\mathcal{H}^*\mathcal{P}_\Omega E_z^\mathcal{P} = \phi_z$ from Remark \ref{remarkinjectivityP}, one can derive that $\mathcal{P}_\Omega\mathcal{T}\mathcal{P}_\Omega v_z= E_z^\mathcal{P}$. Note on the other hand that $E_{q,z} = \mathcal{P}_{q,\Omega}\mathcal{T}_b\mathcal{P}_\Omega v_z$, this allows to prove \eqref{eqn: E^P_z and E^P_qz}.

The equivalent representations of $\lim_{\alpha \to 0}\langle  g_{z,\alpha}, \phi_z  \rangle_{L^2(\mathbb{S}^{d-1})} 
$ can be directly derived using \eqref{Section FM and a LSM main of LSM theorem LSM=FM EzP} and \eqref{eqn: E^P_z and E^P_qz}. This completes the proof.
\end{proof}
It is noted that the above formulation of the indicator function   depends on the simple linear operator $\mathcal{T}_b$   but involves $P_{q,\Omega} \mathcal{T}$ and $E^\mathcal{P}_{q,z}$ which depend on the unknown $q$ more implicitly.  The nonlinear information in Theorem \ref{Section parameter LSM theorem EzP} is more explicit.

\section{Analytic examples} \label{Section analysis inside}

The goal of this section is to provide analytic examples    to shed light on the new interpretation of the linear sampling indicator.

\subsection{The unique projection $E_z^\mathcal{P}$ in spherical and cylindrical symmetric cases}
We shed light on the unique projection $E_z^\mathcal{P}$ in the case of $\Omega$ being a disk (in the case of 2D) and a ball (in the case of 3D). To that purpose we will consider the following equations in     $B_R:=\{x \in \mathbb{R}^d: |x|<R\}$,
 \begin{eqnarray}
(\Delta + k^2)E_z^\mathcal{P} = 0 &\mbox{ in }& B_R, \label{EqDisk1} \\  
(\Delta + k^2) w_z^\mathcal{P} = E_z^\mathcal{P} &\mbox{ in }& B_R, \label{EqDisk2} \\  
w_z^\mathcal{P}=\Phi(\cdot,z)&\mbox{ on }& \partial B_R, \label{EqDisk3} \\  
\frac{\partial  w_z^\mathcal{P} }{\partial \nu}=\frac{\partial\Phi(\cdot,z)}{\partial \nu}&\mbox{ on }& \partial B_R. \label{EqDisk4}
\end{eqnarray}
In the following, $J_n$ and $j_n$ denotes the Bessel function and spherical Bessel function, respectively; $\{Y_n^m(\hat{x}): -n \le m\le n, n =0,1,\cdots\}$ form an orthonormal system in $L^2(S^2)$, and in particular for $\hat{x} = (\sin \theta \cos \phi, \sin \theta \sin \phi, \cos \theta)^T$, 
$$
Y_n^m(\hat{x}) = \sqrt{\frac{(2n+1) (n- |m|!)}{4 \pi (n + |m|!)}} P_n^{|m|} \exp(i m \phi)
$$
where $P_n^m$ denotes the so-called associated Legendre function. For more details, we refer to \cite{ColtonKirsch96}. In two dimensions, let $x = |x| (\cos \theta_x, \sin \theta_x)^T$. 
\begin{proposition}
Equations \eqref{EqDisk1} - \eqref{EqDisk4} admit a solution given by the following expression
\begin{equation*}
E_z ^\mathcal{P}(x)=\left\{
    \begin{array}{cc}
      - \sum_{n=-\infty}^{+\infty} \frac{J_n(k|z|) \exp(-in\theta_z)}{\pi R^2 \big(J_n(kR)^2-J_{n-1}(kR)J_{n+1}(kR) \big)} J_n(k|x|) \exp(in\theta_x),   &  d=2, \\
      & \\
       \sum_{n=0}^{+\infty}\sum_{m=-n}^{+n} \frac{2 j_n(k|z|)\overline{Y_n^m(\hat{z})}}{R^3 \big(j_n(kR)^2-j_{n-1}(kR)j_{n+1}(kR) \big)}j_n(k|x|) Y_n^m(\hat{x}),  & d =3.
    \end{array}
    \right.
\end{equation*}
Moreover, for sufficiently small $R$,
\begin{equation*}
E_z ^\mathcal{P}(x) = \left\{
    \begin{array}{cc}
      -\frac{1}{\pi R^2}( 1+ 2 \sum_{n=1}^{+\infty}(n+1 )\big(\frac{|x||z|}{R^2}\big)^n \cos(n(\theta_x-\theta_z)))+ \mathcal{O}(1),   &  d=2, \\
      &\\
       \frac{1}{R^3}\sum_{n=0}^{\infty} \sum_{m=-n}^{m=n} (2n+3)\big(\frac{|x||z|}{R^2}\big)^n  Y_n^m(\hat{x}) \overline{Y_n^m(\hat{z})} + \mathcal{O}(\frac{1}{R}),  & d =3,
    \end{array} \quad \mbox{when }  R\ll 1
    \right..
\end{equation*}
Given fixed $x$ and $z$, the following asymptotic holds for  sufficiently large $R$,
\begin{equation*}
E_z ^\mathcal{P}(x)=\left\{
    \begin{array}{cc}
      -\frac{k}{2R}J_0(k|x-z|) + \mathcal{O}(\frac{1}{R^2}),  &  d=2, \\
      &\\
       \frac{k }{2\pi R} j_0(k|x-z|) + \mathcal{O}(\frac{1}{R^3}),  & d =3.
    \end{array}
    \right. \quad \mbox{when }  R\gg 1.
\end{equation*}
\end{proposition}
\begin{proof}
\noindent \textit{Two dimensional case $d=2$}. We first consider the two dimensional case. We look for $E_z$ by the following series expansion
$$ 
E_z^\mathcal{P}(x)= \sum_{n=-\infty}^{+\infty} \alpha_n J_n(k|x|) \exp(in\theta_x)
$$
with unknown coefficients $\alpha_n\in \mathbb{C}$ to be determined. From \eqref{EqDisk1} -- \eqref{EqDisk4}, it can be seen that $ w_z^\mathcal{P} $ can be extended to $\mathbb{R}^d\backslash\overline{\Omega}$ as a radiating solution and we can express $ w_z^\mathcal{P} $ by
$$
w_z^\mathcal{P} (x) =- \int_{B_R}\Phi(x,y) E_z^\mathcal{P}(y) \ind y. 
$$
From the series expansion (cf. \cite{ColtonKirsch96})
 \begin{eqnarray*}
\Phi(x,y)  =  \frac{i}{4}  \sum_{n=-\infty}^{+\infty} H_n^{(1)}(k|x|)J_n(k|y|) \exp(in\theta_x)\exp(-in\theta_y) \mbox{ when } |x|>|y|,
\end{eqnarray*}
we can derive that
{\small
\[
 w_z^\mathcal{P} (x) = -\frac{i\pi}{2}  \sum_{n=-\infty}^\infty \alpha_n \exp(in\theta_x)   \Big( H_n^{(1)}(k|x|) \int_{r=0}^{|x|}    J_n(kr)^2 r\ind r   + J_n(k|x|)  \int_{r=|x|}^R H^{(1)}_{n}(kr)J_{n}(kr)r \ind r \Big).    
\]		
}
Matching the boundary conditions \eqref{EqDisk3} -- \eqref{EqDisk4}, we obtain that
$$
 -2\pi   \alpha_n \int_{r=0,R}    J_n(kr)^2 r \ind r   =   J_n(k|z|) \exp(-in\theta_z),
 $$
which together with Lommel's integrals yields
 $$
 \alpha_n  = -\frac{1}{\pi R^2}\frac{J_n(k|z|) \exp(-in\theta_z)}{ J_n(kR)^2-J_{n-1}(kR)J_{n+1}(kR)}.
 $$

 Now for large $R \gg 1$,  we can use the asymptotic expansion of $J_n$ and derive $\alpha_n =-k/(2R)J_n(k|z|) \exp(-in\theta_z)  + \mathcal{O}(\frac{1}{R^2})$  and consequently
 $$
 E_z^{\mathcal
 P}(x)= -\frac{k}{2R}\sum_{n=-\infty}^{+\infty} J_n(k|z|) \exp(-in\theta_z)J_n(k|x|) \exp(in\theta_x) + \mathcal{O}(\frac{1}{R^2})=  -\frac{k}{2R}J_0(k|x-z|) + \mathcal{O}(\frac{1}{R^2}).
$$

 Similarly for small $R \ll 1$ and $|x|,|z| \leq R$, we can use the asymptotic expansion for small argument to derive that  
$$
 E_z^{\mathcal
 P}(x)=-\frac{1}{\pi R^2}( 1+ 2 \sum_{n=1}^{+\infty}(n+1 )(\frac{|x||z|}{R^2})^n \cos(n(\theta_x-\theta_z)))+ \mathcal{O}(1).
 $$

    \noindent \textit{Three dimensional case}. Similarly we look for the following representation
    $$
    E_z^\mathcal{P}(x)= \sum_{n=0}^{+\infty}\sum_{m=-n}^{+n} \alpha_n^m j_n(k|x|) Y_n^m(\hat{x}).
    $$
    where $\hat{x} = x/|x|$.
    From the series expansion of the fundamental solution (cf. \cite{ColtonKirsch96})
    $$
    \Phi(x,y) = \frac{e^{ik|x-y|}}{4\pi |x-y|} = ik\sum_{n=0}^{\infty}\sum_{m=-n}^{n}h_n^{(1)}(k|x|)Y_n^{m}(\hat{x}) j_n(k|y|) \overline{Y_n^{m}(\hat{y})} \mbox{ when } |x|>|y|,
    $$
    and $w_z^\mathcal{P} (x) =- \int_{B_R}\Phi(x,y) E_z^\mathcal{P}(y) \ind y$, we can similarly derive that 
    \begin{multline*}
       w_z^\mathcal{P} (x) = -ik\sum_{n=0}^{\infty}\sum_{m=-n}^{n} \alpha_n^m Y_n^m(\hat{x}) \Big( \int_{r=0}^{|x|}  h_n^{(1)}(k|x|) j_n(kr)^2   r^2 \ind r \\
        + \int_{r=|x|}^{R}  h_n^{(1)}(kr) j_n(k|x|)j_n(kr)r^2 \ind r \Big).  
    \end{multline*}
Matching the boundary conditions yield that
$$
i k \alpha_n^m \int_{r=0,R}   j_n(kr)j_n(kr) r^2dr = i k j_n(k|z|)\overline{Y_n^m( \hat{z})}
$$
Using again the Lommel's integrals and the relation $j_n(kr)=\sqrt{\frac{\pi}{2kr}}J_{n+1/2}(kr)$, one can derive that
$$
\alpha_n^m = \frac{2j_n(k|z|)\overline{Y_n^m(\hat{z})}}{R^3\big(j_n(kR)^2-j_{n-1}(kR)j_{n+1}(kR) \big)}.
$$
From the asymptotic behaviour of $j_n$ \cite{ColtonKirsch96}, we can conclude  for large $R \gg 1$,
$$
E_z^\mathcal{P}(x) = \frac{k j_0(k|x-z|)}{2\pi R} +  \mathcal{O}(\frac{1}{R^3})
$$
and for small $R \ll 1$,
$$
E_z^\mathcal{P}(x) = \frac{1}{R^3}\sum_{n=0}^{\infty} \sum_{m=-n}^{m=n} (2n+3)\big(\frac{|x||z|}{R^2}\big)^n  Y_n^m(\hat{x}) \overline{Y_n^m(\hat{z})} +  \mathcal{O}(\frac{1}{R}).
$$
This completes the proof.
\end{proof}

\subsection{Analytic example: verification of Corollary \ref{Section parameter LSM corollary Born} for constant contrast}
In this section, we give an analytic example to illustrate our main result. In particular, we explicitly verify Corollary \ref{Section parameter LSM corollary Born} in the Born case where the contrast $q$ in $B_R=\{x\in \mathbb{R}^2: |x|<R\}$ is a constant in two dimensions. It is known that the Born far field pattern can be computed explicitly by
\[
u^\infty_b (\hat{x};\hat{\theta};k) = \int_{B_R} e^{ik (\hat{\theta} - \hat{x})\cdot p} q(p) \ind p = q 2\pi R \frac{J_1(kR|\hat{\theta}-\hat{x}|)}{k |\hat{\theta}-\hat{x}|)},
\]
and one can compute the eigensystem of the Born far field operator 
\[
\left( \mathcal{F}_b g \right)(\hat{x})  = \int_{\mathbb{S}^{d-1}}  u^{\infty}_b(\hat{x};\hat{\theta};k) g(\hat{\theta}) \ind s(\hat{\theta}) = \int_{\mathbb{S}^{d-1}}    q 2\pi R \frac{J_1(kR|\hat{\theta}-\hat{x}|)}{k |\hat{\theta}-\hat{x}|)}g(\hat{\theta}) \ind s(\hat{\theta})
\]
via the Jacob-Anger expansion to obtain that
\[
\mathcal{F}_b \Big(\frac{e^{in \theta}}{\sqrt{2\pi}} \Big)= \mu_n \frac{e^{in \theta}}{\sqrt{2\pi}}, \quad \mu_n  = 2\pi^2 k^2 R^2 q \big(J_n(kR)^2-J_{n-1}(kR)J_{n+1}(kR) \big), \quad n=0,\pm 1, \pm 2, \cdots.
\]
This allows us to compute the left hand side of \eqref{Section FM and a LSM main of LSM theorem LSM=FM Born} by
\[
\lim_{\alpha \to 0} \langle  g_{b,z,\alpha}, \phi_z  \rangle_{L^2(\mathbb{S}^{d-1})} = \sum_{n=-\infty}^\infty \frac{1}{\mu_n} |\langle  \phi_z,\frac{e^{in \theta}}{\sqrt{2\pi}}  \rangle_{L^2(\mathbb{S}^{d-1})}|=\sum_{n=-\infty}^{+\infty} \frac{J_n^2(k|z|)  }{k^2 q \pi R^2 \big(J_n(kR)^2-J_{n-1}(kR)J_{n+1}(kR) \big)}  
\]
where we have used that 
\[
\langle  \phi_z,\frac{e^{in \theta}}{\sqrt{2\pi}}  \rangle_{L^2(\mathbb{S}^{d-1})} = \langle  e^{-ik \hat{x}\cdot z},\frac{e^{in \theta_x}}{\sqrt{2\pi}}   \rangle_{L^2(\mathbb{S}^{d-1})} =  \sqrt{2 \pi} i^m J_m(k|z|) e^{im \theta_z}.
\]
For the right hand side of \eqref{Section FM and a LSM main of LSM theorem LSM=FM Born}, one can directly verify that
\[
\frac{1}{k^2}\langle \frac{1}{q} E_z^{\mathcal{P}}, E_z^{\mathcal{P}} \rangle_{L^2(B_R)} = \sum_{n=-\infty}^{+\infty} \frac{J_n^2(k|z|)  }{k^2 q \pi R^2 \big(J_n(kR)^2-J_{n-1}(kR)J_{n+1}(kR) \big)},  
\]
and this shows \eqref{Section FM and a LSM main of LSM theorem LSM=FM Born}, i.e.,
\begin{equation*} 
\lim_{\alpha \to 0} \langle  g_{b,z,\alpha}, \phi_z  \rangle_{L^2(\mathbb{S}^{d-1})}  = \langle \mathcal{T}^{-1}_{b} E_z^{\mathcal{P}},E_z^{\mathcal{P}}  \rangle_{L^2(\Omega)}=\frac{1}{k^2}\langle \frac{1}{q} E_z^{\mathcal{P}}, E_z^{\mathcal{P}} \rangle_{L^2(\Omega)}.
\end{equation*}
 
To conclude the paper, we point out that the methodology to demonstrate the nonlinear information is general and we expect that it can be applicable to other nonlinear inverse problems (cf. \cite{Kirsch21}) including the electric impedance tomography and time domain inverse scattering (cf. \cite{Cakoni2019}). Moreover, since the limit of the imaging indicator is quantified clearly by the nonlinear information, one future direction may be to explore stability for the linear sampling method in similar spirit to the increasing stability (cf. \cite{HI2004,Isakov2020}).

 \appendix

\section{Proof of Theorem \ref{Section operator Omega fac theorem}} \label{section appendix factorization}
\begin{proof}
Define the solution operator $\mathcal{G}: L^2 (\Omega) \to L^2(\mathbb{S}^{d-1})$ by 
$$
\mathcal{G} f := u_f^\infty
$$
where $u_f^\infty$ is the far-field pattern of $u_f$ which is the unique radiating solution to $\Delta u_f + k^2 (1+q)u_f = -k^2q \underline{f}$. Then it follows directly that $\mathcal{F} = \mathcal{G} \mathcal{H}$. To prove the theorem, we show that 
$\mathcal{G} f=\mathcal{H}^* \mathcal{T}$ for any $f\in L^2(\Omega)$. For any $f\in L^2(\Omega)$, let $h = \mathcal{T} f$, then according to the definition of  $\mathcal{T}$ and \eqref{Section operator T_omega def v def}, we have that $ h=\mathcal{T} f   = k^2q f+ k^2 q v|_\Omega$ where
$$
\Delta v + k^2  v = -k^2 q (f+v)=-h.
$$
Note also according to the definition of  $\mathcal{H}^*$, we have that $\mathcal{H}^* h$ is the unique far-field pattern of $ \int_{\Omega} \Phi(\cdot,y) h(y) \ind y$, note that $ \int_{\Omega} \Phi(\cdot,y) h(y) \ind y=v$ (by volume integral representation of $v$), then we have that $\mathcal{H}^* h = v^\infty$. From equation \eqref{Section operator T_omega def v def} and the definition of $\mathcal{G}$, we can conclude that $\mathcal{G} f = \mathcal{H}^* h = \mathcal{H}^* \mathcal{T} f$. This completes the proof.
\end{proof}  

\section{Proof of Proposition \ref{Section operator middle positive definite prop}} \label{Section appendix proof of middle operator}
  \begin{proof}
The proof is standard. For self-completeness we give a brief proof following   \cite{kirsch2008factorization}. Part (1) is obvious. 

Part (2)(3). For any $f \in L^2(\Omega)$, according to \eqref{Section operator T_omega def} and \eqref{Section operator T_omega def v def}
\begin{eqnarray*}
\langle \mathcal{T} f,f\rangle_{L^2(\Omega)} = k^2 \langle  qf+q v|_\Omega,f\rangle_{L^2(\Omega)}.
\end{eqnarray*}
Now we further calculate
\begin{eqnarray*}
-k^2\langle  q v|_\Omega,f\rangle_{L^2(\Omega)} = \langle  v|_\Omega, -k^2qf\rangle_{L^2(\Omega)}  =  \langle  v, \Delta v + k^2(1+q) v\rangle_{L^2(\Omega)},
\end{eqnarray*}
and by further working out
\begin{eqnarray*}
 &&\langle  v, \Delta v + k^2(1+q) v\rangle_{L^2(\Omega)} = \int_{\partial \Omega} \frac{\partial \overline{v}}{\partial \nu} v \ind s -\int_{\Omega} \nabla v \cdot \overline{\nabla v} \ind x + \int_{\Omega} k^2 (1+q) |v|^2 \ind x \\
 &=& \int_{|x|=R} \frac{\partial \overline{v}}{\partial \nu} v \ind s -\int_{|x|<R} \nabla v \cdot \overline{\nabla v} \ind x + \int_{|x|<R} k^2 (1+q) |v|^2 \ind x,
\end{eqnarray*}
we can now conclude that
\begin{eqnarray*}
&&\Im \langle \mathcal{T} f,f\rangle_{L^2(\Omega)} = \Im(k^2\langle  q v|_\Omega,f\rangle_{L^2(\Omega)} ) = \Im(-\langle  v, \Delta v + k^2(1+q) v\rangle_{L^2(\Omega)} ) \\
&=&-\Im \int_{|x|=R} \frac{\partial \overline{v}}{\partial \nu} v \ind s = k \lim_{R\to \infty}\int_{|x|=R} |v|^2 \ind s \ge 0.
\end{eqnarray*}
If $k$ is not an interior transmission eigenvalue and $f \in Y_\Omega$, we show that the above left hand side is strictly larger than zero. Otherwise, if ``='' is true, then $v^\infty$ vanishes so that $v$ vanishes outside $\Omega$ by analytic continuation. Since $f \in Y_\Omega$, then  $f$ satisfies the Helmholtz equation in $\Omega$ in the distributional sense, and thereby $(w=v+f,f)$ is a pair that satisfies the interior transmission eigenvalue problem in Definition \ref{Section operator ITE def}. However $k$ is not an interior transmission eigenvalue thereby  $f$ vanishes. This proves the lemma.

Part (4) follows   from the results in (1)(2)(3) and Lemma \ref{Section appendix coercive lemma} which is a particular case of \cite[Lemma 1.17]{kirsch2008factorization}. This completes the proof.
\end{proof}
 
 \section{Lemma \ref{Section appendix coercive lemma}}
 The following lemma is a particular case of \cite[Lemma 1.17]{kirsch2008factorization} that is needed in our paper.
 \begin{lemma}(\cite[Lemma 1.17]{kirsch2008factorization}) \label{Section appendix coercive lemma}
 Let
$A,A_0: L^2(\Omega) \to L^2(\Omega)$ be linear and bounded operators such that
\begin{enumerate}
\item $\langle A\phi,\phi  \rangle_{L^2(\Omega)} \in \mathbb{C}\backslash i(-\infty,0]$ for all $0\not=\phi \in \overline{ R(\mathcal{H})}$,
\item $\langle A_0\phi,\phi  \rangle_{L^2(\Omega)}$ is real-valued, and there exists $c_{0,A}>0$ with
$$
\langle A_0\phi,\phi  \rangle_{L^2(\Omega)} \ge c_{0,A} \|\phi\|^2_{L^2(\Omega)} \mbox{ for all } \phi \in \mbox{closure} R(\mathcal{H})
$$ 
\item $A-A_0$ is compact.
\end{enumerate}
Then there exists $c_{A}>0$ with
$$
\langle A_0\phi,\phi  \rangle_{L^2(\Omega)} \ge c_{A} \|\phi\|^2_{L^2(\Omega)} \mbox{ for all } \phi \in \overline{ R(\mathcal{H})}.
$$ 
\end{lemma}

\section{Proof of $\mathcal{F}$ is normal} \label{Section appendix F is normal}
\begin{proof}
Proof of that
$$
\mathcal{I} + i \tau \mathcal{F} \mbox{ is unitary, } \quad \tau :=  \left\{
\begin{array}{cc}
\frac{1}{4 \pi}  & d=2   \\
\frac{k}{8 \pi^2}  &  d=3
\end{array}
\right..
$$

We refer to \cite[Theorem 4.4]{kirsch2008factorization} for the three dimensional case. We prove the case when the dimension $d=2$. Let $v_g^i=\mathcal{H} g$ be the Herglotz wave function corresponding to kernel $g$ and   $v_g^s$ be the unique  solution to \eqref{medium us eqn1+2} with $f=v_g^i$. Set $v_g^{s,\infty}$ as the far-field pattern of $v_g^{s}$, this gives that $\mathcal{F} g = v_g^{s,\infty}$. 

We first note from the far-field asymptotic that
\begin{eqnarray*}
&&\lim_{r \to \infty}\int_{|x|=r} \Big(  \frac{\partial v_g^s}{\partial \nu} \overline{v_h^s} -  v_g^s \frac{\partial \overline{v_h^s}}{\partial \nu}    \Big) \ind s_x = \lim_{r \to \infty}\int_{|x|=r} \Big(  ik v_g^s \overline{v_h^s} -  v_g^s  (-ik \overline{v_h^s} )  \Big) \ind s_x \\
&=& 2ik \lim_{r \to \infty}\int_{|x|=r}\bigg( 
\frac{e^{i\frac{\pi}{4}}}{\sqrt{8k\pi}} \frac{e^{ikr}}{\sqrt{r}} v_g^{s,\infty}(\hat{x};\hat{\theta};k) \overline{\frac{e^{i\frac{\pi}{4}}}{\sqrt{8k\pi}} \frac{e^{ikr}}{\sqrt{r}} v_h^{s,\infty}(\hat{x};\hat{\theta};k)}\bigg) \ind s_x \\
&=& 2ik \int_{\mathcal{S}}\bigg( 
\frac{1}{ 8k\pi}  v_g^{s,\infty}(\hat{x};\hat{\theta};k) \overline{  v_h^{s,\infty}(\hat{x};\hat{\theta};k)}\bigg) \ind s_x = \frac{i}{4\pi} \langle \mathcal{F}g,  \mathcal{F}h\rangle.
\end{eqnarray*}
On the other hand, we have from Green's formula that
\begin{eqnarray*}
&&\int_{|x|=r} \Big(  \frac{\partial v_g^s}{\partial \nu} \overline{v_h^s} -  v_g^s \frac{\partial \overline{v_h^s}}{\partial \nu}    \Big) \ind s_x = \int_{\partial \Omega} \Big(  \frac{\partial v_g^s}{\partial \nu} \overline{v_h^s} -  v_g^s \frac{\partial \overline{v_h^s}}{\partial \nu}    \Big) \ind s_x = \int_{\Omega} \Big(  \Delta v_g^s \overline{v_h^s} -  v_g^s \Delta \overline{v_h^s}   \Big) \ind x \\
&=& \int_{\Omega} \Big(  -k^2 ( v_g^s(1+q) +q v_g^i) \overline{v_h^s} -  v_g^s (-k^2) \overline{( v_h^s(1+q) +q v_h^i) }  \Big) \ind x \\
&=&  \int_{\Omega} \Big(  -k^2  q v_g^i \overline{v_h^s} +  k^2 v_g^s  \overline{ v_h^i }  \Big) \ind x,
\end{eqnarray*}
the above two equations allow to have that
\begin{eqnarray*}
\frac{i}{4\pi} \langle \mathcal{F}g,  \mathcal{F}h\rangle = \int_{\Omega} \Big(  -k^2  q v_g^i \overline{v_h^s} +  k^2 v_g^s  \overline{ v_h^i }  \Big) \ind x.
\end{eqnarray*}

Now we compute 
\begin{eqnarray*}
&&\langle \mathcal{F}g,  h\rangle = \langle v_g^{s,\infty},  h\rangle = \int_{\mathbb{S}} \bigg[\int_{\partial \Omega} \bigg(  \frac{\partial e^{-ik\hat{x}\cdot y}}{\partial \nu_y} v_g^s(y)  -e^{-ik\hat{x}\cdot y} \frac{\partial v_g^s(y)}{\partial \nu_y}  \bigg) \ind  y\bigg] \overline{h(\hat{x})} \ind s_{\hat{x}} \\
&=& \int_{\mathbb{S}} \bigg[\int_{  \Omega} \bigg(  -k^2 e^{-ik\hat{x}\cdot y}   v_g^s(y) - e^{-ik\hat{x}\cdot y}  (-k^2 (1+q)v_g^s(y) -k^2 q v_g^i(y))   \bigg) \ind  y\bigg] \overline{h(\hat{x})} \ind s_{\hat{x}} \\
&=&k^2 \int_{\mathbb{S}} \bigg[\int_{  \Omega}  e^{-ik\hat{x}\cdot y} \big(     q v_g^s(y) + q v_g^i(y)   \big) \ind  y\bigg] \overline{h(\hat{x})} \ind s_{\hat{x}} \\
&=& k^2 \int_{  \Omega} \overline{v_h^i(y)} \big(     q v_g^s(y) + q v_g^i(y)   \big) \ind  y,
\end{eqnarray*}
and 
\begin{eqnarray*}
&&\langle \mathcal{F}^* g,  h\rangle = \overline{\langle \mathcal{F} h,  g\rangle}  = k^2 \int_{  \Omega} v_g^i(y) \big(     q \overline{v_h^s(y)} + q \overline{v_h^i(y)}  \big) \ind  y,
\end{eqnarray*}
the above two equations give
\begin{eqnarray*}
\langle \mathcal{F}g,  h\rangle - \langle \mathcal{F}^* g,  h\rangle = \int_{\Omega} \Big(  -k^2  q v_g^i \overline{v_h^s} +  k^2 v_g^s  \overline{ v_h^i }  \Big) \ind x.
\end{eqnarray*}
Hence we can conclude that
\begin{eqnarray*}
\frac{i}{4\pi} \langle \mathcal{F}g,  \mathcal{F}h\rangle = \langle \mathcal{F}g,  h\rangle - \langle \mathcal{F}^* g,  h\rangle.
\end{eqnarray*}
It is then directly verified that
\begin{eqnarray*}
\frac{i}{4\pi} \langle  \mathcal{F}^*\mathcal{F}g,  h\rangle = \langle \mathcal{F}g,  h\rangle - \langle \mathcal{F}^* g,  h\rangle =\frac{i}{4\pi} \langle \mathcal{F} \mathcal{F}^* g,  h\rangle
\end{eqnarray*}
so that
$$
(\mathcal{I}+ \frac{i}{4\pi} \mathcal{F})^*(\mathcal{I}+ \frac{i}{4\pi} \mathcal{F} )= \mathcal{I} + \frac{i}{4\pi} \big( \mathcal{F}- \mathcal{F}^*-\frac{i}{4\pi} \mathcal{F}^*\mathcal{F}\big) = \mathcal{I},
$$
and consequently $(\mathcal{I}+ \frac{i}{4\pi} \mathcal{F} )(\mathcal{I}+ \frac{i}{4\pi} \mathcal{F})^*=\mathcal{I}$ and $\mathcal{F}^*\mathcal{F}=\mathcal{F}\mathcal{F}^*$.
\end{proof}


\section{Proof of Lemma \ref{Section operator mu/|mu| phase interval}} \label{Section appendix proof of accumulation interval}
\begin{proof}
Proof of the following statement: if $k$ is not an interior transmission eigenvalue, then the  accumulation point of $\frac{\mu_n}{|\mu_n|}$ cannot be $-1$ and $\frac{\mu_n}{|\mu_n|} = e^{i \eta_n}$ with $\eta_n \in [0,\pi-2\eta_\delta)$ with $\eta_\delta \in (0,\pi/2]$. We follow the proof in  \cite[pp. 190]{cakoni2016qualitative}.

If $k$ is not an interior transmission eigenvalue, then $\mathcal{P}_\Omega \mathcal{T} \mathcal{P}_\Omega: Y_\Omega \to Y_\Omega$ is coercive and so that $\mathcal{F}$ is injective (as is seen from \eqref{Section operator Omega fac coercive factorization}), thereby $\frac{\mu_n}{|\mu_n|}$ lies on a unit circle.   Now for any $(\zeta_n,\mu_n)$ where
 $
\mathcal{F} \zeta_n = \mu_n \zeta_n
 $,
we derive
\begin{eqnarray*}\label{Section appendix F eigs distribution eqn1}
\mu_n =\langle \mathcal{F} \zeta_n,\zeta_n \rangle_{L^2(\mathbb{S}^{d-1})} = \langle \mathcal{F} \zeta_n,\zeta_n \rangle_{L^2(\mathbb{S}^{d-1})} =  \langle   \mathcal{T}   \mathcal{H}\zeta_n,\mathcal{H}\zeta_n \rangle_{L^2(\Omega)},
\end{eqnarray*}
taking the imaginary part of the above equation    we can derive that 
$$
\Im \mu_n \overset{Prop~\eqref{Section operator middle positive definite prop}}{>}0,
$$
this shows that $\frac{\mu_n}{|\mu_n|}$ lies on the upper unit circle in the complex plane.

Since $\mathcal{F}$ is compact, then $\frac{\mu_n}{|\mu_n|}$ accumulates to some point on the unit circle. We now show that  the  accumulation point of $\frac{\mu_n}{|\mu_n|}$ cannot be $-1$. In this regard, we set $ \varphi_n = \frac{1}{\sqrt{|\mu_n|}}\mathcal{H}\zeta_n$ for all $n=0,1,\cdots$, it immediately follows that
\begin{equation} \label{Section appendix F eigs distribution eqn2}
\langle \mathcal{T} \varphi_n,\varphi_n \rangle_{L^2(\Omega)} = \frac{\mu_n}{|\mu_n|}.
\end{equation}
Note that $\mathcal{T}$ is coercive, then it follows that $\{\varphi_n\}_{n=0}^\infty$ is a bounded sequence in the separable Hilbert space $L^2(\Omega)$ whereby there exists a subsequence $\varphi_{n_m}$ that weak$^*$ converges to $\varphi \in L^2(\Omega)$. By a density argument and the weak$^*$ convergence, one has that $\varphi \in Y_\Omega$. Since $ \mathcal{C}$  is compact, then   there exists a subsequence, still denoted as  $\varphi_{n_m}$ without loss of any rigor, such that $\mathcal{C} \varphi_{n_m}$ converges strongly to $\mathcal{C} \varphi$ as $m\to \infty$. Now we can have that  
$$
\langle   \mathcal{C}  \varphi_{n_m}, \varphi_{n_m} \rangle_{L^2(\Omega)} = \langle   \mathcal{C}  (\varphi_{n_m} -\varphi), \varphi_{n_m} \rangle_{L^2(\Omega)}  + \langle   \mathcal{C}  \varphi, \varphi_{n_m} \rangle_{L^2(\Omega)} \to \langle   \mathcal{C}  \varphi, \varphi \rangle_{L^2(\Omega)}.
$$
As a consequence, taking the imaginary part of \eqref{Section appendix F eigs distribution eqn2}
\begin{eqnarray*}
\Im \frac{\mu_{n_m}}{|\mu_{n_m}|} = \Im \langle \mathcal{C} \varphi_{n_m},\varphi_{n_m} \rangle 
\end{eqnarray*}
and as $m\to \infty$, we have that 
$$
0=\Im \langle \mathcal{C}  \varphi, \varphi \rangle_{L^2(\Omega)}=\Im \langle \mathcal{T}  \varphi, \varphi \rangle_{L^2(\Omega)},
$$
as a consequence $\varphi \in Y_\Omega$ must vanish since $k$ is not an interior transmission eigenvalue.

From \eqref{Section appendix F eigs distribution eqn2}, we can derive that 
\begin{eqnarray*}
  \frac{\mu_{n_m}}{|\mu_{n_m}|}  - \langle   \mathcal{C}  \varphi_{n_m}, \varphi_{n_m} \rangle_{L^2(\Omega)} = \langle   \mathcal{T}_{b}   \varphi_{n_m}, \varphi_{n_m} \rangle_{L^2(\Omega)} \ge k^2 q_{\inf} \|\varphi_{n_m}\|^2_{L^2(\Omega)}.
\end{eqnarray*}
Note that the left hand side goes to $-1$ as $m \to \infty$,  however the right hand is always non-negative which yields a contradiction. This completes the proof. 
\end{proof}

 \section{Proof of Lemma \ref{Section FM and a LSM lemma phiz S Omega} } \label{Section appendix FM and a LSM lemma phiz S Omega}
\begin{proof}
The proof follows \cite[Theorem 4.6]{kirsch2008factorization}.  If $z \in \Omega$, then there exists some mall $\epsilon>0$ such that $\overline{B(z,\epsilon)} \subset \Omega$, and we define  $w_z(x)= \chi_{|x-z|}\Phi(x,z)$ as in the Lemma. As a consequence $w_z$ and $\Phi(\cdot,z)$ have the same Cauchy data so that by the Green's formula
\begin{eqnarray*}
w_z(x) &=& \int_{\partial \Omega} \Big(\Phi(x,y) \frac{\partial w_z(y)}{\partial \nu_y} -  w_z(y) \frac{\partial \Phi(x,y)}{\partial \nu_y} \Big) \ind s_y - \int_\Omega \Big( \Delta_y w_z(y) +k^2 w_z(y) \Big) \Phi(x,y) \ind y \\
&=&\int_{\partial \Omega} \Big(\Phi(x,y) \frac{\partial \Phi(y,z)}{\partial \nu_y} -  \Phi(y,z)\frac{\partial \Phi(x,y)}{\partial \nu_y} \Big) \ind s_y   - \int_\Omega \Big( \Delta_y w_z(y) +k^2 w_z(y) \Big) \Phi(x,y) \ind y,
\end{eqnarray*}
since the first integral on the right hand side vanishes by applying the Green's formula outside $\Omega$ (and by noting that $\Phi(\cdot,x)$ and $\Phi(\cdot,z)$ are both radiating solutions to the Helmholtz equation), then we obtain that
\begin{eqnarray*}
w_z(x) =  - \int_\Omega \Big( \Delta_y w_z(y) +k^2 w_z(y) \Big) \Phi(x,y) \ind y =\int_\Omega   \Phi(x,y) E_z(y)\ind y,
\end{eqnarray*}
whereby their far-field patterns are the same, i.e.,
\begin{eqnarray*}
e^{-ik\hat{x} \cdot z} = \Phi^\infty(\cdot,z)  = w_z^\infty =   \int_\Omega   e^{-i k \hat{x} \cdot y} E_z(y)\ind y= \left( \mathcal{H}^* E_z \right)(\hat{x}).
\end{eqnarray*}

If $z \not \in \Omega $ and $\phi_z  =\mathcal{H}^* \varphi$, then by Rellich's lemma \cite[Lemma 1.2]{kirsch2008factorization} and unique continuation
$$
\int_\Omega   \Phi(x,y) \varphi(y)\ind y =\Phi(x,z), \quad \forall x \in \mathbb{R}^d \backslash \overline{\Omega}, \quad x \not=z.
$$ 
This is a contradiction since the right hand side has a singularity at $x=z$. This completes the proof.
\end{proof}
 
\bibliographystyle{SIAM}

\end{document}